\documentclass[leqno]{siamltex}
\usepackage{amssymb}
\usepackage{amsbsy}
\usepackage{newlfont}
\usepackage{amsmath}
\usepackage{psfrag}
\usepackage{color}  
\usepackage{xspace,float,stackrel,enumerate,subfigure}
\usepackage{bm}
\usepackage{bbm}
\usepackage{mathrsfs} 
\usepackage{algorithm}
\usepackage{algpseudocode}
\usepackage{remarks}
\usepackage{mydef}
\usepackage{graphicx}
\usepackage{listings}
\usepackage[numbers]{natbib}
\let\cite=\citet
\newcommand{\FIGS}{.}
\newcommand{\FCODES}{.}
\begin{document}
\newcommand\footnotemarkfromtitle[1]{%
\renewcommand{\thefootnote}{\fnsymbol{footnote}}%
\footnotemark[#1]%
\renewcommand{\thefootnote}{\arabic{footnote}}}

\title{Fast estimation from above of the maximum  \\ 
wave speed in the Riemann problem \\  
for the Euler equations\footnotemark[1]}

\author{Jean-Luc Guermond\footnotemark[2]
\and Bojan Popov\footnotemark[2]}
\date{Draft version \today}

\maketitle

\renewcommand{\thefootnote}{\fnsymbol{footnote}} \footnotetext[1]{
  This material is based upon work supported in part by the National
  Science Foundation grant DMS-1217262, by the Air
  Force Office of Scientific Research, USAF, under grant/contract
  number FA99550-12-0358, and the Army Research Office, under grant number W911NF-15-1-0517. 
Submitted to Journal of Computational Physics, November 6 2015;
Accepted, May 26 2016} \footnotetext[2]{Department of Mathematics, Texas
  A\&M University 3368 TAMU, College Station, TX 77843, USA.}
\renewcommand{\thefootnote}{\arabic{footnote}}

\begin{abstract}
  This paper is concerned with the construction of a fast algorithm
  for computing the maximum speed of propagation in the Riemann
  solution for the Euler system of gas dynamics with the co-volume
  equation of state.  The novelty in the algorithm is that it stops
  when a guaranteed upper bound for the maximum speed is reached with
  a prescribed accuracy. The convergence rate of the algorithm is
  cubic and the bound is guaranteed for gasses with the co-volume equation
  of state and the heat capacity ratio $\gamma$ in the range $(1,5/3]$.
\end{abstract}



\pagestyle{myheadings} \thispagestyle{plain} \markboth{J.L. GUERMOND,
  B. POPOV}{Fast estimation of the maximum speed  in a Riemann problem}

\section{Introduction}
The objective of this paper is to propose a fast algorithm to
estimate from above the maximum wave speed in the Riemann
problem for the Euler equations of gas dynamics. This quantity, or an
approximation thereof, is used in many numerical methods to
approximate the solution of the compressible Euler equations using
various representations: finite volumes, discontinuous Galerkin,
continuous finite elements, \etc, see \eg \cite[Eq.~(9)]{Rusanov_1961}
or \cite[Eq.~(2.6b)]{Harten_Lax_VanLeer_1983}. The motivation for the
present work comes from a multidimensional finite element technique
recently proposed in \cite{Guermond_Popov_Hyp_2015}.  This method is
explicit and uses continuous finite elements on unstructured grids in
any space dimension. The artificial viscosity in the method is defined
so that having an upper bound on the maximum speed of propagation in
the one-dimensional Riemann problems guaranties that all the entropy
inequalities are satisfied and the algorithm is invariant domain
preserving in the sense of
\cite{Chueh_Conley_Smoller,Hoff_1985,Frid_2001}, \ie the density and
the internal energy are nonnegative and the specific entropy satisfies
a local minimum principle. It is also shown therein that the closer
the upper bound on the maximum wave speed the larger the admissible
CFL. We stress here that it is not the entire solution of the Riemann
problem that is required to ensure the above properties, but only a
guaranteed upper bound on the maximum wave speed. Standard Riemann
solvers, either approximate or exact, are designed to give an
approximation of the solution at the interface, and this in general
requires solving for intermediate states in the Riemann fan.  This task
is far more computationally intensive than estimating the maximum wave
speed of the Riemann fan. Note in passing that traditional estimates
of the maximum wave speed in ideal gases, which consists of taking
$\max(|u_L|+a_L,|u_R|+a_R)$, where $a$ is the speed of sound and $u$
is the velocity, could either be wrong or be an overestimate thereof (a
counterexample is produced in the appendix~\ref{Sec:counter_example}),
see \eg \cite[Eq.~(3.2)]{Kurganov_Tadmor_JCP_2000} or
\cite[\S10.5.1]{Toro_2009}. In conclusion, we claim that
$\max(|u_L|+a_L,|u_R|+a_R)$ is not an upper bound on the maximum wave
speed, and solving for the intermediate states, as done in traditional
Riemann solvers, is expensive and is not necessary to ensure that
invariant domains are preserved, as established in
\citep{Guermond_Popov_Hyp_2015,Guermond_Popov_Saavedra_Yong_ALE_2016}.

The novelty of the present work is the construction of a fast
algorithm for computing the maximum wave speed in the Riemann problem
for the Euler equations with the co-volume equation of state (which
includes ideal gases). One important feature of the algorithm is that
it terminates when an upper bound for the maximum speed is obtained
with a prescribed tolerance. The algorithm has a cubic convergence rate and
the upper bound is guaranteed for gasses with co-volume equation of
state and a heat capacity ratio $1<\gamma\leq 5/3$. We have obtained
$10^{-15}$ accuracy in at most three steps in all the numerical
experiments we have done with the proposed algorithm. We stress
  here that traditional Riemann solvers based on Newton's method on
  the intermediate pressure $p^*$ do not guarantee the upper bound on
  the maximum wave speed; actually most of these solvers converge to
  $p^*$ from below and thereby underestimate the maximum wave speed.
  Another original result of the present work is an inexpensive guaranteed
  upper bound of the maximum wave speed given in
  Remark~\ref{Remark:inexpensive_estimate}. This estimate is actually
  used to initialize the iterative algorithm.

This paper is organized as follows. We introduce some notation and
collect general statements about the one-dimensional Riemann problem
in \S\ref{Sec:Preliminaries}. The main result of this section is the
well known Proposition~\ref{Prop:lambda_max}. We introduce additional
notation in \S\ref{Sec:Compute_lambda_max} and recall the expression
for the extreme wave speeds of the 1-wave and the 3-wave.  We
introduce the algorithm to compute a guaranteed upper bound on the
maximum wave speed in \S\ref{Sec:Lambda_max}. It is shown in
Theorem~\ref{Thm:guaranteed_estimate} that the algorithm terminates in
finite time and delivers a guaranteed upper bound up to any prescribed
threshold. An important result that makes the convergence of the
method cubic and guarantees the upper bound is stated in
Theorem~\ref{Thm:RR_curve_below}. The gap condition proved in
Lemma~\ref{Lem:gap} is essential to prove that the algorithm
terminates in finite time. Both Theorem~\ref{Thm:RR_curve_below} and
Lemma~\ref{Lem:gap} are original to the best of our knowledge.  The
performance of the algorithm is tested in \S\ref{Sec:numerical};
it is shown in \S\ref{Sec:overhead} that when used in a
  numerical solver for the compressible Euler equations proposed in
  \citep{Guermond_Popov_Hyp_2015},
  the overhead induced by the proposed iterative method is minimal.
Additional theoretical statements for the co-volume equation of state
and counter-examples showing that $\max(|u_L|+a_L,|u_R|+a_R)$ may be
sometimes significantly smaller and sometimes
significantly bigger than the actual maximum wave speed of the
Riemann problem are reported in Appendices~\ref{Sec:Appendix} and
\ref{Sec:counter_example}. A source code is provided in
  Appendix~\ref{Sec:source_code}.

\section{Preliminaries} \label{Sec:Preliminaries} We introduce
notations and discuss the notion of Riemann problem in this
section. The main result is the well known
Proposition~\ref{Prop:lambda_max}. The reader who is already familiar
with Riemann problems and Proposition~\ref{Prop:lambda_max} can skip
this section and go directly to \S\ref{Sec:Compute_lambda_max}.

\subsection{Formulation of the problem}
Consider the compressible
Euler equations
\begin{equation}
\partial_t \bc +\DIV(\bef(\bc)) =0, \quad 
\bc=\left(\begin{matrix}
\rho \\
\bbm \\ 
E
\end{matrix}\right),\qquad
\bef(\bc)=
\left(
\begin{matrix}
\bbm \\ 
 \bbm {\otimes} \frac{\bbm}{\rho} + p\polI \\ 
\frac{\bbm}{\rho} (E+p) 
\end{matrix}\right),
\label{Euler_equations}
\end{equation}
where the independent variables are the density $\rho$, the momentum
vector field $\bbm$ and the total energy $E$.  The velocity vector
field $\bu$ is defined by $\bu:=\bbm/\rho$ and the internal energy
density $e$ by $e:=\rho^{-1}E-\frac12\|\bu\|_{\ell^2}^2$, where
$\|\cdot\|_{\ell^2}$ is the Euclidean norm. The quantity $p$ is the
pressure.  The symbol $\polI$ denotes the identity matrix in
$\polR^{d\CROSS d}$.
In this paper, we only consider the so-called {\em co-volume} gasses
obeying the co-volume Equation Of State (EOS),
\begin{equation}
p(1-b\rho)=(\gamma-1)e\rho,
\end{equation} 
with $b\ge 0$; the case $b=0$ corresponds to an ideal gas. The
constant $b$ is called the co-volume and $\gamma>1$ is the ratio of
specific heats. Sometimes, the co-volume EOS is called the Noble-Abel
EOS. We refer to \cite[Chapter 1.2]{Toro_2009}, \cite{Baibuz_1986} and
\cite{Johnston_2005} for more details on these EOS and the related
thermodynamics.

In the context of the method proposed in \cite{Guermond_Popov_Hyp_2015}, 
we consider the
following one-dimensional Riemann problem:
\begin{equation}
 \label{def:Riemann_problem} 
  \partial_t \bc + \partial_x (\bef(\bc)\SCAL\bn)=0, 
\quad  (x,t)\in \Real\CROSS\Real_+,\qquad 
\bc(x,0) = \begin{cases} \bc_L, & \text{if $x<0$} \\ \bc_R,  & \text{if $x>0$}, \end{cases}.
\end{equation}
where $\bn$ is any unit vector in $S^d(0,1)$.  The solution to this
problem is also invoked in many Riemann-solver-based Godunov type
methods, see \eg
\cite{Toro_Titarev_2006,Bouchut_Morales_2009,Castro_Gallardo_Marquina_2014,Balsara_Dumbser_Abgrall_2014}. We
stress that we are only interested in estimating from above the
maximum wave speed in \eqref{def:Riemann_problem}. It is shown in
\citep{Guermond_Popov_Hyp_2015} that having an upper bound on the
maximum speed of propagation of the one-dimensional Riemann problem
guaranties that the first-order algorithm described in
\citep{Guermond_Popov_Hyp_2015} is invariant domain preserving in the
sense of \cite{Chueh_Conley_Smoller,Hoff_1985,Frid_2001}, and that it
satisfies all the entropy inequalities.

The problem \eqref{def:Riemann_problem} is hyperbolic, since
$\partial_\rho p(\rho,s)$ is positive, \ie the Jacobian of
$\bef(\bc)\SCAL \bn$ is diagonalizable with real eigenvalues.  It is
well known in the case of ideal and co-volume gases with $\gamma>1$
that \eqref{def:Riemann_problem} has a unique (physical) solution,
which we henceforth denote $\bc(\bn,\bu_L,\bu_R)$, see \cite[Chapter
4.7]{Toro_2009}.

Note that the admissibility condition for the left and right states is
$0<1-b\rho_L,1-b\rho_R <1$.  A simple but lengthy verification shows
that the exact solution of the Riemann problem in all possible cases
stays admissible across the entire Riemann fan, that is $1-b\rho> 0$.
Being unaware of a reference for this result, we give a proof in the
appendix for completeness, see
Proposition~\ref{Prop:co_volume_admissible}.

\subsection{Structure of the Riemann problem}
The multidimensional Riemann problem \eqref{def:Riemann_problem} was
first described in the context of dimension splitting schemes in two
space dimensions in \cite[p.~526]{Chorin_1976}. The general case is
treated in \cite[p.~188]{Colella_1990}, see also \cite[Chapter
4.8]{Toro_2009}.  We make a change of basis and introduce
$\bt_1,\ldots,\bt_{d-1}$ so that $\{\bn,\bt_1,\ldots,\bt_{d-1}\}$
forms an orthonormal basis of $\polR^d$. With this new basis we have
$\bbm=(m, \bbm^{\perp})\tr$, where $m:=\rho u$, $u:=\bu\SCAL\bn$,
$\bbm^{\perp}\!:=\rho(\bu\SCAL \bt_1,\ldots,
\bu\SCAL\bt_{d-1}):=\rho\bu^\perp$. The projected equations are
\begin{equation}
\partial_t \bc +\partial_{x}(\bn\SCAL\bef(\bc)) =\mathbf{0}, \quad 
\bc=\left(\begin{matrix}
\rho \\
m\\
\bbm^{\perp} \\ 
E
\end{matrix}\right),\qquad
\bn\SCAL\bef(\bc)=
\left(
\begin{matrix}
m\\
\tfrac{1}{\rho}m^2 +p \\ 
u \bbm^{\perp} \\ 
u (E+ p)
\end{matrix}\right).
\label{Euler_projected}
\end{equation}

Using $\rho$, $u$, $\bu^\perp$ and the specific entropy
$s$ as dependent variables, the above problem can be rewritten
\begin{equation}\left\{
\begin{aligned}
&\partial_t \rho + \partial_x (\rho u) =0 \\
&\partial_t u + u \partial_x (u) + \rho^{-1}\partial_xp(\rho,s) =0 \\
&\partial_t \bu^\perp + u\partial_x (\bu^\perp) =0\\
&\partial_t s + u \partial_x (s) =0,
\end{aligned}\right.
\end{equation} 
and the  Jacobian is
\[
\left(\begin{matrix} 
u                        & \rho  & \mathbf{0}\tr  & 0 \\
\rho^{-1}\partial_\rho p   & u     & \mathbf{0}\tr & \rho^{-1}\partial_s p \\
\mathbf{0}             & \mathbf{0}&  u \polI   & \mathbf{0} \\  
0                      & 0 &     \mathbf{0}\tr   & u   
\end{matrix}\right).
\]
The eigenvalues  are $\lambda_1=u-\sqrt{\partial_\rho p(\rho,s)}$, with multiplicity 1, 
$\lambda_2=\cdots=\lambda_{d+1}=u$, with multiplicity $d$, and 
$\lambda_{d+2}=u+\sqrt{\partial_\rho p(\rho,s)}$, with multiplicity 1. One key observation
is that the Jacobian does not depend on $\bbm^\perp$, see \cite[p.~150]{Toro_2009}.
As a consequence the solution of the Riemann problem with data $(\bc_L,\bc_R)$,
is obtained in two steps. 

\subsubsection{First step} We solve  
the one-dimensional Riemann problem 
\begin{equation}
\partial_t\left(\begin{matrix}
\rho \\
m\\
\mathcal E
\end{matrix}\right)+ \partial_x
\left(
\begin{matrix}
m\\
\tfrac{1}{\rho}m^2 +p \\ 
\tfrac{m}{\rho} (\mathcal{E}+ p)
\end{matrix}\right)=0, \quad \text{with} \quad
p(1-b\rho)=(\gamma-1)\left(\mathcal{E} - \tfrac{m^2}{2\rho}\right)
\label{1D_Euler}
\end{equation}
with data $\bc_L:=(\rho_L,\bbm_L\SCAL\bn,\mathcal{E}_L)\tr$,
$\bc_R:=(\rho_R,\bbm_R\SCAL\bn,\mathcal{E}_R)\tr$, where $\mathcal{E}
= E - \frac12 \frac{\|\bbm^\perp\|_{\ell^2}^2}{\rho}$. The
one-dimensional Riemann problem \eqref{1D_Euler} is strictly
hyperbolic and all the characteristic fields are either genuinely
nonlinear or linearly degenerate. Any Riemann problem of this type
with $n$ fields has a unique self-similar weak
solution in Lax's form for any initial data such that
$\|\bu_L-\bu_R\|_{\ell^2}\le \delta$, see \cite{Lax_1957_II} and
\cite[Thm~5.3]{Bressan_2000}. In particular there are $2n$ numbers
\begin{equation}
\lambda_1^-\le \lambda_1^+ \le  \lambda_2^-\le \lambda_2^+ \le \ldots \le 
\lambda_n^-\le \lambda_n^+,
\end{equation}
defining up to $2m+1$ sectors (some could be empty) in the $(x,t)$ plane:
\begin{equation}
\frac{x}{t}\in (-\infty,\lambda_1^-), \quad
\frac{x}{t}\in (\lambda_1^-,\lambda_1^+), \ldots, \quad
\frac{x}{t}\in (\lambda_n^-,\lambda_n^+), \quad \frac{x}{t}\in (\lambda_n^+,\infty). 
\end{equation}
The Riemann solution is $\bu_L$ in the sector
$\frac{x}{t}\in (-\infty,\lambda_1^-)$ and $\bu_R$ in the last sector
$\frac{x}{t}\in (\lambda_n^+,\infty)$.  The solution in the other
sectors is either a constant state or an expansion, see
\cite[Chap.~5]{Bressan_2000}.  The sector
$\lambda_1^- t < x < \lambda_n^+ t$, $0<t$, is henceforth referred to as the
Riemann fan.  The key result that we are going to use is that there is
a maximum speed of propagation
$\lambda_{\max}(\bn,\bu_L,\bu_R):=\max(|\lambda_1^-|,|\lambda_n^+|)$
such that for $t\ge 0$ we have
\begin{equation}
\bu(x,t)= \begin{cases}
\bu_L, & \text{if $x \le -t \lambda_{\max}(\bn,\bu_L,\bu_R)$}\\
\bu_R, & \text{if $x \ge   t \lambda_{\max}(\bn,\bu_L,\bu_R)$}.
\end{cases} \label{finite_speed}
\end{equation}
In the special case of the one dimensional Euler equations of gas
dynamics \eqref{1D_Euler} with the co-volume EOS, we have $n=3$, the smallness assumption
on the Riemann data is not needed, see \cite[Chap. 4]{Toro_2009}, and
the Riemann fan is composed of three waves only: (i) two genuinely
nonlinear waves, $\lambda_i^{\pm}$, $i\in \{1,3\}$, which are either
shocks (in which case $\lambda_i^{-} = \lambda_i^{+}:= \lambda_i$) or
rarefaction waves; (ii) one linearly degenerate middle wave which is a
contact discontinuity, $\lambda_2^{-} = \lambda_2^{+}=:u^*$.

\subsubsection{Second step} We complete the full solution of the
Riemann problem \eqref{Euler_projected} by determining $\bbm^\perp$.
We compute $\bbm^\perp$ by solving $\partial_t \bbm^\perp
+ \partial_x(u \bbm^\perp)=0$. The solution is composed of up to four
states: $\bbm^\perp_L$, $\bbm^{\perp,*}_L$, $\bbm^{\perp,*}_R$,
$\bbm^\perp_R$,
\begin{equation}
\bbm=
\begin{cases} 
\bbm^\perp_L & \text{if $x\le \lambda_1^+ t$}, \\
\bbm^{\perp,*}_L & \text{if $\lambda_1^+ t\le x\le \lambda_2 t$},\\
\bbm^{\perp,*}_R & \text{if $\lambda_2 t\le x\le \lambda_3^- t$}, \\
\bbm^{\perp}_R & \text{if $\lambda_3^- t \le x$}, \\
\end{cases}
\end{equation}
where $\bbm^{\perp,*}_L$ is such that $\bbm^{\perp,*}_L =
\bbm^{\perp}_L $ if $\lambda_1^{-} \ne \lambda_1^{+}$ (\ie if the leftmost
wave is a rarefaction) or $\bbm^{\perp,*}_L$ is given by the
Rankine-Hugoniot condition $u_L\bbm^\perp_L - u^*\bbm^{\perp,*}_L =
\lambda_1 (\bbm^\perp_L -\bbm^{\perp,*}_L)$ otherwise (\ie if the
leftmost wave is a shock).  Similarly $\bbm^{\perp,*}_R$ is computed
as follows: $\bbm^{\perp,*}_R = \bbm^{\perp}_R $ if $\lambda_3^{-} \ne
\lambda_3^{+}$ (\ie the rightmost wave is a rarefaction) or
$\bbm^{\perp,*}_R$ is given by the Rankine-Hugoniot condition
$u_R\bbm^\perp_R - u^*\bbm^{\perp,*}_R = \lambda_3 (\bbm^\perp_R
-\bbm^{\perp,*}_R)$ otherwise (\ie the rightmost wave is a
shock). The Rankine-Hugoniot condition is automatically
satisfied across the contact wave $u^*(\bbm^{\perp,*}_L -
\bbm^{\perp,*}_R) = \lambda_2 (\bbm^{\perp,*}_L -\bbm^{\perp,*}_R)$
since $u^*:=\lambda_2$. Note that the solution given in
\cite[\S3.2.4,\S4.8]{Toro_2009} is correct only if the two extreme
waves (\ie the 1-wave and the 3-wave) are both rarefactions.

\subsubsection{Maximum wave speed} 
The bottom line of the above argumentation is that the organization of
the Riemann fan is entirely controlled by the solution of
\eqref{1D_Euler} and therefore we have
\begin{proposition} \label{Prop:lambda_max}
  In the case of gases obeying the co-volume equation of state,
  the maximum wave speed in \eqref{Euler_projected} is
\begin{equation}
\lambda_{\max}(\bc_L,\bc_R)=
\max((\lambda_1^-(\bc_L,\bc_R))_-,(\lambda_3^+(\bc_L,\bc_R))_+),
\end{equation}
$z_-=\max(0,-z)$, $z_+=\max(0,z)$, and
$\lambda_1^-(\bc_L,\bc_R)$,
$\lambda_3^+(\bc_L,\bc_R)$ are the two extreme wave speeds in
the  Riemann problem~\eqref{1D_Euler} with data
$(\bc_L,\bc_R)$.
\end{proposition}

The goal of this paper is to propose a fast algorithm to estimate
accurately from above the maximum speed of propagation
$\lambda_{\max}(\bc_L,\bc_R)$.  This program is achieved by estimating
$\lambda_1^-(\bc_L,\bc_R)$ from below and
$\lambda_3^+(\bc_L,\bc_R)$ from above.

\section{Computation of $\lambda_1^-(\bc_L,\bc_R)$  and
$\lambda_3^+(\bc_L,\bc_R)$}
\label{Sec:Compute_lambda_max}
We compute $\lambda_1^-(\bc_L,\bc_R)$ and
  $\lambda_3^+(\bc_L,\bc_R)$ in this section. Most of the material is
  quoted from \cite[Chap.~4]{Toro_2009}. We henceforth abuse the
  notation by using the symbol $\bc$ for the primitive variables, \ie
  $\bc=(\rho,u,p)\tr$, instead of using the conservative variables as defined above.

\subsection{Riemann data with vacuum} When vacuum is present in
  the left or in the right state, the Riemann solution is composed of 
two waves only: one rarefaction on the non-vacuum side and a contact discontinuity 
on the vacuum side. The left and right maximum wave speeds are given by the following 
expressions:
\begin{equation}
\lambda_1^-(\bc_L,\bc_R)=u_L -a_L,\quad
\lambda_2(\bc_L,\bc_R)=u_L+\frac{2a_L(1-b\rho_L)}{\gamma-1},
 \label{estimate_speeds_Euler_vacuum_right}
\end{equation}
with $\bc_L=(\rho_L,u_L,p_L)\tr$ and $\bc_R=(0,u_R,0)\tr$ for vacuum on the
right, and
\begin{equation}
\lambda_2(\bc_L,\bc_R)=u_R-\frac{2a_R(1-b\rho_R)}{\gamma-1},\quad
\lambda_3^+(\bc_L,\bc_R)=u_R+a_R,
 \label{estimate_speeds_Euler_vacuum_left}
\end{equation}
with $\bc_L=(0,u_L,0)\tr$ and $\bc_R=(\rho_R,u_R,p_R)\tr$ for vacuum on the
left. See \S4.6.1, \S4.6.2 and \S4.7.1 in \cite{Toro_2009} for the
details. Then $\lambda_{\max}(\bc_L,\bc_R)=
\max((\lambda_1^-)_{-},(\lambda_2)_+)$ in the first case and
$\lambda_{\max}(\bc_L,\bc_R)= \max((\lambda_2)_{-},(\lambda_3^+)_+)$
in the second case.

\subsection{Riemann data without vacuum}
We now restrict ourselves to the case where both states,
$\bc_L$ and $\bc_R$, are not vacuum states, \ie $\rho_L,\rho_R>0$,
$e_L,e_R \ge 0$, and admissible states $1-b\rho_L, 1-b\rho_R > 0$.
A simple but lengthy verification shows that the exact solution
  of the Riemann problem in all possible cases is admissible for all
  times, that is $1-b\rho> 0$. This result is proved in
  Proposition~\ref{Prop:co_volume_admissible} in the
  appendix~\ref{Sec:Appendix}. In the numerical applications we have
  in mind, see \cite{Guermond_Popov_Hyp_2015}, only averages of the
  exact Riemann solution over the {\em entire} Riemann fan are
  invoked.  Therefore, if the left and the right states are admissible
  non-vacuum states, the average of the exact Riemann solution over
  the Riemann fan is an admissible non-vacuum state even if the exact
  solution contains vacuum. Since the method in
  \citep{Guermond_Popov_Hyp_2015} preserves all the convex invariant
  sets of the Riemann problem and the set $\{\rho>0,\ 1-b\rho>0\}$ is
  convex and invariant, we conclude that if the initial data is
  admissible and does not contain a vacuum state, the numerical method
  in \citep{Guermond_Popov_Hyp_2015} produces only admissible
  non-vacuum states at every time step.

The  no vacuum condition $0<\rho_L,\rho_R$ and the admissibility conditions
$0<1-b\rho_L,1-b\rho_R$, $0\le e_L,e_R$ imply that
$p_L,p_R\in[0,\infty)$.  Then the local sound speed is given by
$a_Z=\sqrt{\frac{\gamma p_Z}{\rho_Z(1-b\rho_Z)}}$ where the index $Z$
is either $L$ or $R$.  We introduce the following notation
$A_Z:=\frac{2(1-b\rho_Z)}{(\gamma+1)\rho_Z}$,
$B_Z:=\frac{\gamma-1}{\gamma+1}p_Z$ and the functions
\begin{align}
\phi(p)&:=f(p,L)+f(p,R)+u_R-u_L\\
f(p,Z)&:=\begin{cases} (p-p_{Z})\left(\frac{A_{Z}}{p+B_{Z}}\right)^{\frac12} 
& \text{if $p\ge p_{Z}$},\\
\frac{2a_{Z}(1-b\rho_Z)}{\gamma-1}\left(\left(\frac{p}{p_{Z}}\right)^{\frac{\gamma-1}{2\gamma}}-1\right) 
& \text{if  $p < p_{Z}$},
\end{cases}\label{Eq:rarefaction_shock}
\end{align}
where again $Z$ is either $L$ or $R$.  Let $a^0_Z$ be the speed of
sound for the ideal gas, and let $A^0_Z$, $B^0_Z$, $\phi^0(p)$ and $f^0(p,Z)$
be the above defined quantities in the ideal gas case, \ie we take
$b=0$ in all definitions. Then we have that
$a_Z=\frac{a^0_Z}{\sqrt{1-b\rho_Z}}$,
$f(p,Z)=f^0(p,Z)\sqrt{1-b\rho_Z}$ and
\begin{equation}
\phi(p)=f^0(p,L)\sqrt{1-b\rho_L} + f^0(p,R)\sqrt{1-b\rho_R} + u_R-u_L.
\end{equation}
It is shown in \cite[Chapter~4.3.1]{Toro_2009} (see also
\cite[Eq.~(5.36)]{Bressan_2000}) that the functions
$f^0(p,L),f^0(p,R)\in C^2( \Real_+;\Real)$ are monotone increasing and
concave down. Therefore the function $\phi(p)\in C^2( \Real_+;\Real)$
is also monotone increasing and concave down.  It can also be shown
that the weak third derivative is non-negative and locally bounded.
Observe that $\phi(0) =
u_R-u_L-\frac{2a^0_L\sqrt{1-b\rho_L}}{\gamma-1}-\frac{2a^0_R\sqrt{1-b\rho_R}}{\gamma-1}$.
Therefore, $\phi$ has a unique positive root if and only if $\phi(0)
<0$, \ie
\begin{equation}\label{no_vacuum}
u_R-u_L<\frac{2a^0_L\sqrt{1-b\rho_L}}{\gamma-1}+\frac{2a^0_R\sqrt{1-b\rho_R}}{\gamma-1}.
\end{equation}
This is the well known non-vacuum condition in the case of ideal gas
($b=0$ above), see \cite[(4.40), p.~127]{Toro_2009}. We henceforth
denote this root by $p^*$, \ie $\phi(p^*)=0$. We conventionally set
$p^*=0$ if \eqref{no_vacuum} does not hold.  It can be shown that,
whether there is formation of vacuum or not, the two extreme wave
speeds $\lambda_1^-(\bc_L,\bc_R)$ and
$\lambda_3^+(\bc_L,\bc_R)$ enclosing the Riemann fan are
\begin{align}
\lambda_1^-(\bc_L,\bc_R)=u_L - a_L\left(
1+\frac{\gamma+1}{2\gamma}\left(\frac{p^*-p_L}{p_L}\right)_+
\right)^\frac12, 
\label{estimate_speed_1_Euler}
\\
\lambda_3^+(\bc_L,\bc_R)=u_R + a_R\left(
1+\frac{\gamma+1}{2\gamma}\left(\frac{p^*-p_R}{p_R}\right)_+
\right)^\frac12 ,
\label{estimate_speed_3_Euler}
\end{align}
where $z_+:=\max(0,z)$. 

\begin{remark}[Two rarefaction waves] \label{Rem:RR_case} Note that if
  $\phi(p_L)>0$ then $p_L>p^*$, by monotonicity of $\phi$, thereby
  implying that $\lambda_1^-(u_L,u_R)=u_L - a_L$ in this
  case. Similarly if $\phi(p_R)>0$, then $p_R>p^*$ and
  $\lambda_3^+(u_L,u_R)=u_R + a_R$. This observation means that there
  is no need to compute $p^*$ to estimate
  $\lambda_{\max}(\bc_L,\bc_R)$ when
  $\phi(\min(p_L,p_R))>0$. This happens when the two extreme
  waves are rarefactions.
  Noticing that $p^*$ does  not need to be evaluated in this case is important
  since traditional techniques to compute $p^*$ in this situation may
  require a large number of (unnecessary) iterations, see
  \cite[p.~128]{Toro_2009}.  This is particularly true when
  \eqref{no_vacuum} is violated, since in this case there is a formation
  of a vacuum state.
\end{remark}

\section{Accurate estimation of $\lambda_{\max}$ from above} 
\label{Sec:Lambda_max}
In this
section we present an algorithm for computing an accurate lower
bound on $\lambda_1^-(\bc_L,\bc_R)$ and an accurate upper bound on 
$\lambda_3^-(\bc_L,\bc_R)$. This is done by estimating accurate lower and upper bounds on
the intermediate pressure state $p^*$.

\subsection{Elementary waves}
If the exact solution of the Riemann problem
contains two rarefaction waves, \ie $p^*\le \min(p_L,p_R)$, no
computation of $p^*$ is needed, see Remark~\ref{Rem:RR_case}.

Let us define $p_{\min}:=\min(p_L,p_R)$, $p_{\max}:=\max(p_L,p_R)$ and
let us assume that $\phi(p_{\min})\le 0$. Note that if
$\phi(p_{\min})=0$, then $p^*=p_{\min}$ and nothing needs to be
done. We now assume that $p^*>p_{\min}$ and we define
\begin{multline}
\phi_R(p)=\frac{2a_{L}(1-b\rho_L)}{\gamma-1}\left(\left(\frac{p}{p_{L}}\right)^{\frac{\gamma-1}{2\gamma}}-1\right) 
+\frac{2a_{R}(1-b\rho_R)}{\gamma-1}\left(\left(\frac{p}{p_{R}}\right)^{\frac{\gamma-1}{2\gamma}}-1\right) \\
+ u_R -u_L.
\end{multline}
Note that $\phi_R$ is monotone increasing and concave down. We also
have the following result (see also left panel in Figure~\ref{Fig:phi_phir}).
\begin{theorem} \label{Thm:RR_curve_below} Assume $\gamma \in
  (1,\frac53]$.  For any $p\ge 0$, the graph of $(p,\phi(p))$ is above
  the graph of $(p,\phi_R(p))$; more precisely, $\phi_R(p)= \phi(p)$
  for all $p\in [0, p_{\min}]$ and $\phi_R(p)< \phi(p)$ for all $p\in
  (p_{\min}, \infty)$.
\end{theorem}
\begin{proof} Note that the two curves $(p,\phi(p))$and
  $(p,\phi_R(p))$ coincide if $p\le p_{\min}$ because both are the sum
  of the two rarefaction curves and the constant $u_R-u_L$. If
  $p_{\min} < p\le p_{\max}$ the $(p,\phi(p))$ curve is the sum of one
  rarefaction curve, one shock curve starting from $p=p_{\min}$ and
  the constant $u_R-u_L$. If $p \ge p_{\max}$ the $(p,\phi(p))$ curve
  is the sum of two shock curves and the constant $u_R-u_L$, see
  \eqref{Eq:rarefaction_shock}.  Now we invoke
  Lemma~\ref{Lem:fS_GT_fR} twice to complete the proof, once with
  $p^0=p_{\min}$ ($\rho^0$ being the associated density)
  and once with $p^0=p_{\max}$ ($\rho^0$ being the 
  associated density).
\end{proof}

\begin{lemma} \label{Lem:fS_GT_fR} Let $p^0>0$ and $\rho^0$ such that
  $0<1-b\rho^0 <1$. Assume that $1 < \gamma\le \frac53$. We define the
  shock curve passing through $p^0$ by
\[
f_S(p) =
(p-p^0)\sqrt{\frac{2}{(\gamma+1)\rho^0}}\left(p+\frac{\gamma-1}{\gamma+1}p^0\right)^{-\frac12}
\sqrt{1-b\rho^0}
\]
and the rarefaction curve by
\[
f_R(p) = \frac{2\sqrt{\frac{\gamma p^0}{\rho^0}}}{\gamma-1}
\left(\left(\frac{p}{p^0}\right)^{\frac{\gamma-1}{2\gamma}}-1\right)\sqrt{1-b\rho^0}.
\]
Then $f_R(p) < f_S(p)$ for any $p> p^0$ and $f_R(p^0)= f_S(p^0)$, \ie
the shock curve is above the rarefaction curve.
\end{lemma}
\begin{proof}
  We rescale both the shock and the rarefaction curves for 
  $p\ge p^0$ by introducing the variable $x=p/p^0$ and set $\tilde f_S(x):=f_S(p)$, 
$\tilde f_R(x) := f_R(p)$. The two curves now
  are now
\[
\tilde f_S(x)=\frac{ \sqrt{2p^0(1-b\rho^0)} }{\sqrt{(\gamma+1)\rho^0}}
\frac{(x-1)}{
\sqrt{x+\frac{\gamma-1}{\gamma+1}}},\quad
\tilde f_R(x)=\frac{2\sqrt{\gamma p^0(1-b\rho^0)}}{(\gamma-1)\sqrt{\rho^0}}
\left(x^{\frac{\gamma-1}{2\gamma}}-1\right).
\]
Note that $\tilde f_S(1)= \tilde f_R(1)=0$. We assume from now on that $x>1$ and we
want to show that $\tilde f_S(x)> \tilde f_R(x)$ for any $x> 1$ if $\gamma\le 5/3$.
Instead of proving this directly, we consider the function
\[
g(x)= 
\frac{\rho^0(\gamma+1)
\left(x+\frac{\gamma-1}{\gamma+1}\right)}{2p^0(1-b\rho^0)\left(x^{\frac{\gamma-1}{2\gamma}}-1\right)^2}
\left( \tilde f_S(x)^2-\tilde f_R(x)^2\right)
\]
and we will show that $g(x)> 0$ for any $x > 1$ if $\gamma\le
5/3$. We change variable again and set
$y:=x^{\frac{\gamma-1}{2\gamma}}$ with $\tilde g(y):=g(x)$. Then setting $y^\alpha=x$ with
$\alpha=2+\frac{2}{\gamma-1}$ we have
\begin{equation}\label{eq:phi_y}
\tilde g(y)=\left(\frac{y^\alpha-1}{y-1}\right)^2 - \alpha((\alpha-1)y^\alpha+1).
\end{equation} 
 We
rearrange the terms in \eqref{eq:phi_y} to get
\[
\tilde g(y)=\left(\frac{y^\alpha-1}{y-1}-\frac12\alpha(\alpha-1)(y-1)\right)^2 
- \frac14\alpha^2(\alpha-1)^2(y-1)^2-\alpha^2, \qquad \forall y>1.
\]
Using a Taylor expansion of $y^\alpha$ at $y=1$ for any $y> 1$
and $\alpha\ge 4$, we obtain the inequality
\[
\frac{y^\alpha-1}{y-1}\ge \alpha +\frac12\alpha(\alpha-1)(y-1) +
\frac16\alpha(\alpha-1)(\alpha-2)(y-1)^2.
\]
Using this inequality in \eqref{eq:phi_y}, we have that
\[
\tilde g(y)\ge \left(\alpha + \frac16\alpha(\alpha-1)(\alpha-2)(y-1)^2\right)^2
- \frac14\alpha^2(\alpha-1)^2(y-1)^2 - \alpha^2
\]
which is equivalent to
\[
\tilde g(y)\ge \left(\frac16\alpha(\alpha-1)(\alpha-2)(y-1)^2\right)^2
+ \frac{\alpha^2(\alpha-1)}{12}(\alpha-5)(y-1)^2.
\]
Therefore, we infer that $\tilde g(y)>0$ for any $y>1$ provided $\alpha\ge
5$. Note that the condition $\alpha \in [5,\infty)$ is equivalent to $\gamma \in (1,\frac53]$.
 Hence we conclude that $\tilde f_S(x)> \tilde f_R(x)$ for any $x> 1$ if
$1<\gamma\le 5/3$.
\end{proof}

\begin{remark}[Physical range of $\gamma$]
  Note that the $\gamma$-law usually assumes that
  $\gamma=\frac{M+2}{M}$, where $M\ge 3$ is the number of degrees of
  freedom of the molecules composing the gas. We have $M=3$ for
  monatomic gases and $M=5$ for diatomic gases. Therefore, the
  physical range of $\gamma$ for $M\in [3,\infty)$ is $\gamma \in
  (1,\frac53]$, which happens to be exactly the range of application 
of Lemma~\ref{Lem:fS_GT_fR}.
\end{remark}

\begin{remark}[Non physical range of $\gamma$]
  In the non-physical range $\gamma > 5/3$ it can be shown via Taylor
  series argument that there is $x_0=x_0(\gamma)>1$ such that $f_S(x_0) <
  f_R(x_0)$.  Therefore, the statements of both
  Theorem~\ref{Thm:RR_curve_below} and Lemma~~\ref{Lem:fS_GT_fR} are
  false if $\gamma\in(5/3,+\infty)$.
\end{remark}

\subsection{The algorithm to estimate $\lambda_{\max}$}
We now continue with the construction of an algorithm for computing
the intermediate pressure $p^*$, keeping in mind that the quantity we
are after is $\lambda_{\max}$.  Recall that we only consider the
case $\phi(p_{\min})<0$. Both functions $\phi$ and $\phi_R$ are
strictly monotone increasing and
$\lim_{p\to\infty}\phi(p)=\lim_{p\to\infty}\phi_R(p)=+\infty$,
therefore they  each have a unique zero. The zero of $\phi$ is $p^*$ and
we denote the zero of $\phi_R$ by $\tilde{p}^*$.  
The zero of $\phi_R$ is easy to compute 
\begin{equation}
\tilde{p}^*= \left(
\frac{a_L^0\sqrt{1-b\rho_L}+a_R^0\sqrt{1-b\rho_R}-\frac{\gamma-1}{2}(u_R-u_L)}
{ a_L^0\sqrt{1-b\rho_L}\ p_L^{-\frac{\gamma-1}{2\gamma}}
+ a_R^0\sqrt{1-b\rho_R}\ p_R^{-\frac{\gamma-1}{2\gamma} } }
\right)^{\frac{2\gamma}{\gamma-1}}
\end{equation}
and is referred to in the literature as the two-rarefaction
approximation to $p^*$, see for example equation (4.103) in
\cite[Chapter 4.7.2]{Toro_2009}. 
\begin{lemma} \label{Lem:pstar_less_tilde_pstar}
 We have $p^*<\tilde{p}^*$ in
the physical range of $\gamma$, $1<\gamma \le \frac53$.
\end{lemma}
\begin{proof}
This is an easy consequence of Theorem~\ref{Thm:RR_curve_below}.
To the best of our knowledge, this result, which is important to
establish accurate a priori error estimates on $p^*$, is new.
\end{proof}

We now propose an iterative process composed of two algorithms
  that constructs two sequences $(p_1^k,p_2^k)_{k\ge 0}$ such that
  $p_1^k\le p^* \le p_2^k$ for all $k\ge 0$ and $\lim_{k\to
    +\infty}p_1^k= p^* = \lim_{k\to +\infty}p_2^k$.  Before going
  through the details of the algorithms, we describe a stopping
  criterion that guarantees accuracy on the maximum wave speed.  Let
  $k\ge 0$ and assume that $p_1^k\le p^* \le p_2^k$. Then using
  \eqref{estimate_speed_1_Euler}-\eqref{estimate_speed_3_Euler}, we
  have $v_{11}^k \le \lambda_1^- \le v_{12}^k$ and $v_{31}^k \le
  \lambda_3^+ \le v_{32}^k$, where
\begin{subequations} \label{def_v11_v12_v31_v32}
\begin{align}
v_{11}^k = u_L - a_L\left(
1+\tfrac{\gamma+1}{2\gamma}\big(\tfrac{p_2^k-p_L}{p_L}\big)_+
\right)^{\frac12}\!\!\!,\quad&
v_{12}^k=u_L - a_L\left(
1+\tfrac{\gamma+1}{2\gamma}\big(\tfrac{p_1^k-p_L}{p_L}\big)_+
\right)^\frac12\!\!\!, 
\\
v_{31}^k = u_R + a_R\left(
1+\tfrac{\gamma+1}{2\gamma}\big(\tfrac{p_1^k-p_R}{p_R}\big)_+
\right)^\frac12\!\!\!,\quad& 
v_{32}^k= u_R + a_R\left(
1+\tfrac{\gamma+1}{2\gamma}\big(\tfrac{p_2^k-p_R}{p_R}\big)_+
\right)^\frac12\!\!\!,
\end{align}
\end{subequations}
and we have $\lambda_{\min}^k < \lambda_{\max}  \le \lambda_{\max}^k$ for any
$k\ge 0$ with the definitions
\begin{equation}
\lambda_{\min}^k := (\max((v_{31}^k)_+,(v_{12}^k)_-))_+,\qquad
  \lambda_{\max}^k := \max((v_{32}^k)_+,(v_{11}^k)_-). \label{def_of_lambda_max_min}
\end{equation} 
\begin{lemma} \label{Lem:stopping_criterion}
  Let $\epsilon>0$, and assume that $p_1^k\le p^* \le p_2^k$ and
  $\lambda_{\min}^k>0$ then
\begin{equation}
\left(\frac{\lambda_{\max}^k}{\lambda_{\min}^k}-1 \le \epsilon\right) 
\Longrightarrow \left(|\lambda_{\max}^k - \lambda_{\max}| \le \epsilon\lambda_{\max}\right) .
\end{equation}
\end{lemma}
\begin{proof} The inequalities $p_1^k\le p^* \le p_2^k$ together with
  \eqref{def_v11_v12_v31_v32} and \eqref{def_of_lambda_max_min} implies
that  $\lambda_{\min}^k < \lambda_{\max} \le \lambda_{\max}^k$, which in turn gives
\begin{align*}
\frac{|\lambda_{\max}^k-\lambda_{\max}|}{\lambda_{\max}} 
&= \frac{\lambda_{\max}^k-\lambda_{\max}}{\lambda_{\max}} 
= \frac{\lambda_{\max}^k}{\lambda_{\max}} - 1 
&& \text{since  $\lambda_{\max}< \lambda_{\max}^k$}\\
& \le \frac{\lambda_{\max}^k}{\lambda_{\min}^k} - 1 \le \epsilon
&&  \text{since  $\lambda_{\min}^k< \lambda_{\max}$}.
\end{align*}
This completes the proof.
\end{proof}

The process that we propose consists of two algorithms: the first one
generates an initial guess $(p_1^0,p_2^0)$ such that
$p_1^0<p^*<p_2^0$; the second one iterates until the stopping
criterion described in Lemma~\ref{Lem:stopping_criterion} is
satisfied.  The initialization is described in
Algorithm~\ref{initialization_algorithm}.

\begin{algorithm}[H]
\caption{Initialization}
\label{initialization_algorithm}
\begin{algorithmic}[1]
\Require $\epsilon$
\Ensure $p_1^0,p_2^0$
\State Set $p_{\min} = \min(p_L,p_R)$, $p_{\max} = \max(p_L,p_R)$
\If{$\phi(p_{\min})\ge 0$}  \State Set $p^*=0$ and compute $\lambda_{\max}$ using
\eqref{estimate_speed_1_Euler}-\eqref{estimate_speed_3_Euler} \Return \EndIf
\If{$\phi(p_{\max})= 0$} \State $p^*=p_{\max}$ and compute $\lambda_{\max}$ using
\eqref{estimate_speed_1_Euler}-\eqref{estimate_speed_3_Euler} \Return \EndIf
\If{$\phi(p_{\max}) < 0$} \State Set $p_1^0 = p_{\max}$ and $p_2^0=\tilde{p}^*$ 
\Comment{This guarantees that $p_1^0 < p^* < p_2^0$}
\Else
\State Set $p_1^0 = p_{\min}$ and $p_2^0=\min(p_{\max},\tilde{p}^*)$ 
\Comment{This guarantees that $p_1^0 < p^* < p_2^0$}
\EndIf 
\State Compute $\lambda_{\max}^0$, $\lambda_{\min}^0$, \textbf{if}
($\lambda_{\min}>0$ and $(\lambda_{\max}^0/\lambda_{\min}^0-1\le \epsilon$)) \textbf{return} \label{Exit_alg1}
\State Set $p_1^0:=\max(p_1^0, p_2^0-\phi(p_2^0)/\phi'(p_2^0))$
\Comment{Improve $p_1$ with one Newton step}
\label{Newtow_step_initialization}
\State Proceed to Algorithm~\ref{update_algorithm} with $(p_1^0,p_2^0)$
\end{algorithmic}
\end{algorithm}

Very often the initialization process described in
  Algorithm~\ref{initialization_algorithm} is accurate enough to be
  used in applications.  Line~\ref{Exit_alg1} in the algorithm checks
  the accuracy of the initialization in terms of relative error on
  $\lambda_{\max}$.

\begin{remark}[Non-iterative estimate] \label{Remark:inexpensive_estimate}
   It is possible to use Algorithm~\ref{initialization_algorithm} alone without iterating 
since $p_2^0$ is an upper bound on $p^*$ and
\begin{equation}
v_{11}^0 \le \lambda_1^{-}  < \lambda_3^{+} \le v_{32}^0.
\end{equation}
Hence, an inexpensive and guaranteed upper bound on $\lambda_{\max}$
is $\max((v_{11}^0)_{-},(v_{32}^0)_+)$.
\end{remark}

\begin{remark}[Extra Newton step]
Note that step~\ref{Newtow_step_initialization} in
Algorithm~\ref{initialization_algorithm} is a Newton iteration. This
step is optional, but we nevertheless include it to correct the bias
introduced by the computation of $\tilde{p}^*$. Our experience is that
$\tilde{p}^*$ is very often much closer to $p^*$ than both $p_{\min}$
and $p_{\max}$. The concavity of $\phi$ guaranties that $p_2^0-\phi(p_2^0)/\phi'(p_2^0)<
p^*$, whence $p_1^0 < p^* < p_2^0$ as desired.
\end{remark}

Given two positive numbers $p_1, p_2$, we now construct two quadratic
polynomials $P\du(p)$ and $P\dd(p)$ such that $P\du(p)$ interpolates
$\phi$ at the points $p_1, p_2, p_2$ and $P\dd(p)$ interpolates $\phi$
at the points $p_1,p_1,p_2$:
\begin{subequations}
\begin{align}
P\dd(p) &:= \phi(p_1) + \phi[p_1,p_1](p-p_1) + \phi[p_1,p_1,p_2](p-p_1)^2, \\
P\du(p) &:= \phi(p_2) + \phi[p_2,p_2](p-p_2) + \phi[p_1,p_2,p_2](p-p_2)^2.
\end{align}
\end{subequations}

Here we use the standard divided difference notation where repeating a
point means that we interpolate the function and its derivative at the
said point. We abuse the notation by omitting the index $k$ for the
two polynomials $P\du$ and $P\dd$.

\begin{figure}
\centerline{
\includegraphics[width=0.4\textwidth]{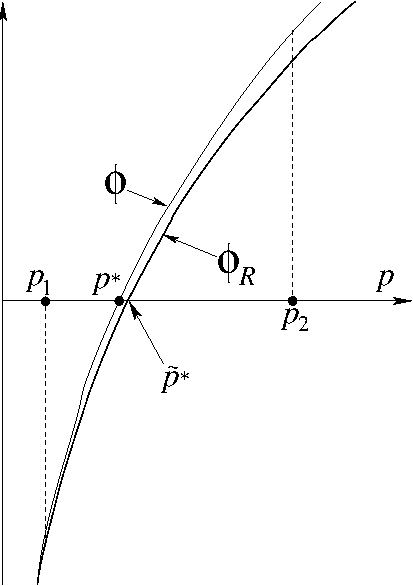}\hfil
\includegraphics[width=0.4\textwidth]{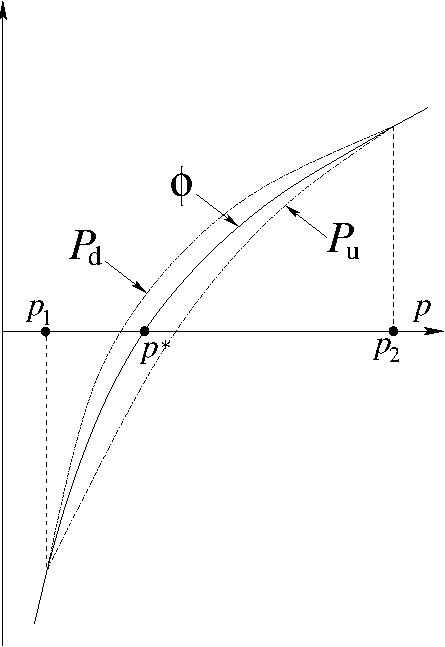}
}
\caption{Left: $\phi(p)=\phi_R(p)$ for all
  $p\le p_1$ and $\phi_R(p)< \phi(p)$ for all $p\in (p_{\min},
  \infty)$. Right: Quadratic polynomials $P_u(p)$ and $P_d(p)$.}
\label{Fig:phi_phir}
\end{figure}

\begin{lemma} \label{Lem:roots} $P\du$ and $P\dd$ have each a unique zero over the
  interval $(p_1,p_2)$ denoted $p\du(p_1,p_2)$ and $p\dd(p_1,p_2)$,
  respectively:
\begin{subequations}\label{roots}
\begin{align} 
p\dd(p_1,p_2)&=p_1  
- \frac{2\phi(p_1)}{\phi'(p_1)+\sqrt{\phi'(p_1)^2-4\phi(p_1)\phi[p_1,p_1,p_2]}}\label{roots_a}\\
p\du(p_1,p_2)&=p_2   
- \frac{2\phi(p_2)}{\phi'(p_2)+\sqrt{\phi'(p_2)^2-4\phi(p_2)\phi[p_1,p_2,p_2]}}.\label{roots_b}
\end{align} 
\end{subequations}
 and the following holds for any $p\in(p_1,p_2)$:
\begin{equation}
P\du(p) < \phi(p) < P\dd(p),\qquad \forall p\in (p_1,p_2)\label{Eq:quad_zeros}
\end{equation}
which implies that $p_1<p\dd(p_1,p_2)<p^*<p\du(p_1,p_2)<p_2$.
\end{lemma}
\begin{proof}
It is a standard result in approximation theory that
\begin{subequations}\label{interpolation_error}
\begin{align} 
\phi(p) - P\du(p) &= \phi[p_1,p_2,p_2,p](p-p_1)(p-p_2)^2\\
\phi(p) - P\dd(p) &= \phi[p_1,p_1,p_2,p](p-p_1)^2(p-p_2)
\end{align} 
\end{subequations}
where $\phi[p_1,p_2,p_2,p]$ and $\phi[p_1,p_1,p_2,p]$ are
divided differences.  For completeness we recall that $f[x] = f(x)$
and given $x_0\le \ldots \le x_n$ we have $f[x_0,\ldots, x_n] =
\frac{1}{n!}f^{(n)}(x_0)$ if $x_0= \ldots = x_n$ and
$f[x_0,\ldots,x_n] =\frac{f[x_0,\ldots,x_{n-1}]
  -f[x_1,\ldots,x_{n}]}{x_0-x_n}$ otherwise.  Moreover we define
$f[x_{\sigma(0)},\ldots, x_{\sigma(n)}] = f[x_0,\ldots, x_n]$ for any
$\sigma\in \calS^{n+1}$ where $\calS^{n+1}$ is the set all the
permutations over the set $\{0,\ldots,n\}$. It is known that for any
$x_0,\ldots,x_n$ we have $f[x_0,\ldots,x_n]=\frac{1}{n!}f^{(n)}(\xi)$
for some $\xi\in [\min(x_0,\ldots,x_n) , \max(x_0,\ldots,x_n)]$. In
the case at hand, we know that $\phi'''(\xi)>0$ for any $\xi>0$,
then \eqref{Eq:quad_zeros} is a simple consequence of
\eqref{interpolation_error}.
Both quadratic polynomials $P\du$ and $P\dd$ are concave down, and
both are negative at $p=p_1$ and positive at $p=p_2$; hence they each
have a unique zero in the interval $(p_1,p_2)$, which we denote
$p\du(p_1,p_2)$ and $p\dd(p_1,p_2)$ respectively.  Moreover, the
inequality \eqref{Eq:quad_zeros} implies that
$p\dd(p_1,p_2)<p^*<p\du(p_1,p_2)$.  The proof is complete.
\end{proof}

The result of Lemma~\ref{Lem:roots} is illustrated in the right panel
of Figure~\ref{Fig:phi_phir}.
The algorithm that we propose proceeds as follows: given a pair
$(p_1^k,p_2^k)$, compute $(p_1^{k+1},p_2^{k+1})$ such that $p_1^{k+1}
= p\dd(p_1^k,p_2^k)$ and $p_2^{k+1} = p\du(p_1^k,p_2^k)$ for $k\ge
0$.  Owing to Lemma~\ref{Lem:roots}, we have $p_1^k\le p_1^{k+1} \le
p^* \le p_2^{k+1} \le p_2^{k}$ and the convergence rate of the
iteration process is cubic.

\begin{algorithm}[H]
\renewcommand{\algorithmicrequire}{\textbf{Input:}}
\renewcommand{\algorithmicensure}{\textbf{Output:}}
\caption{Computation of $\lambda_{\max}$}
\label{update_algorithm}
\begin{algorithmic}[1]
\Require $p_1^0,p_2^0$, $\epsilon$
\Ensure $\lambda_{\max}$
\While{$\mathtt{true}$} 
\State Compute $\lambda_{\max}^k$ and $\lambda_{\min}^k$
\If{$\lambda_{\min}^k>0$} \Comment{Must happen owing to gap condition} \label{first_if}
 \If{$\frac{\lambda_{\max}^k}{\lambda_{\min}^k} -1 \le \epsilon$}
 \label{second_if}
 \State \textbf{exit infinite loop}
 \EndIf
\EndIf
\If{$\phi(p_1^k) >0$ or $\phi(p_2^k) <0$ } \label{third_if} 
\Comment{Check for roundoff error}
\State \textbf{exit infinite loop} 
\EndIf
\State $p_1^{k+1} =  p\dd(p_1^k,p_2^k)$  \label{Jacobi_1}
\State $p_2^{k+1} =  p\du(p_1^k,p_2^k)$  \label{Jacobi_2}
\EndWhile
\State $\lambda_{\max} = \lambda_{\max}^k$ \label{output}
\end{algorithmic}
\end{algorithm}

\begin{theorem} \label{Thm:guaranteed_estimate}
For every $\epsilon>0$,
there exists $k(u_L,a_L,u_R,a_R,\gamma,\epsilon)$ such that
Algorithm~\ref{update_algorithm} terminates when 
$k=k(u_L,a_L,u_R,a_R,\gamma,\epsilon)$, and in this case 
\begin{equation} \label{Eq:Thm:guaranteed_estimate}
|\lambda_{\max}^k - \lambda_{\max} |\le
\epsilon \lambda_{\max},
\end{equation}\ie the relative error on $\lambda_{\max}$
is guaranteed to be bounded by $\epsilon$.
\end{theorem}
\begin{proof}
  Owing the gap Lemma~\ref{Lem:gap}, there is $c(\gamma)>0$ such that
  $\lambda_3^+ -\lambda_1^- \ge c(\gamma) (a_{\max} + a_{\min})$,
  which in turn implies that $d:=\lambda_3^+ -\lambda_1^->0$ owing to
  the hyperbolicity condition $\min(a_{\max},a_{\min})>0$. Hence
  $\lambda_{\max}= \max((\lambda_3^+)_+,(\lambda_1^-)_-) \ge
  \frac{d}{2}>0$, (recall that $\min(x_-,y_+)\ge \frac{|x-y|}{2}$).
  Since both sequences $(p_1^k)_{k\ge 0}$ and $(p_2^k)_{k\ge 0}$
  converge to $p^*$, there is $k_0(u_L,a_L,u_R,a_R,\gamma,\epsilon)$
  such that $\lambda_{\min}^k>\frac12 \lambda_{\max}\ge \frac14 d >0$
  for any $k\ge k_0(u_L,a_L,u_R,a_R,\gamma,\epsilon)$. Hence the
  condition of the \textbf{if} statement in line~\ref{first_if} of
  Algorithm~\ref{update_algorithm} is achieved for any $k\ge
  k_0(u_L,a_L,u_R,a_R,\gamma,\epsilon)$. Likewise there is
  $k(u_L,a_L,u_R,a_R,\gamma,\epsilon)\ge
  k_0(u_L,a_L,u_R,a_R,\gamma,\epsilon)$ such that the condition of the
  \textbf{if} statement in line~\ref{second_if} holds true since
  $\lim_{k\to +\infty} \frac{\lambda_{\max}^k}{\lambda_{\min}^k} =
  1$. Hence, when $k=k(u_L,a_L,u_R,a_R,\gamma,\epsilon)$ the algorithm
  terminates and the estimate \eqref{Eq:Thm:guaranteed_estimate} is
  guaranteed by Lemma~\ref{Lem:stopping_criterion}.  This completes
  the proof.
\end{proof}

\begin{remark}[Jacobi vs. Seidel iterations] Note that
  steps~\ref{Jacobi_1} and~\ref{Jacobi_2} in
  Algorithm~\ref{update_algorithm} are of Jacobi type.  The Seidel
  version of this algorithm is $p_1^{k+1} = p\dd(p_1^k,p_2^k)$,
  $p_2^{k+1} = p\du(p_1^{k+1},p_2^k)$.
\end{remark}

\begin{remark}[Roundoff errors]
  The algorithm converges so fast and is so accurate that it may
  happen that either the test $\phi(p_1^{k})>0$ or the test
  $\phi(p_2^{k})<0$ turns out to be true due to rounding errors. This
  then causes the discriminant in \eqref{roots_b} to be negative
  thereby producing NaN. To avoid the roundoff problem  one must
  check the sign of $\phi(p_1^{k})$ and $\phi(p_2^{k})$ before
  computing $p_1^{k+1}$ and $p_2^{k+1}$ (see line \ref{third_if} in
  Algorithm~\ref{update_algorithm}). If $\phi(p_1^{k})>0$ then
  $p^*=p_1^{k}$ up to roundoff errors and if $\phi(p_2^{k})<0$ then
  $p^*=p_2^{k}$ up to roundoff errors.
\end{remark}

\subsection{Estimation of $\lambda_1^-$ and $\lambda_3^+$} In
  some applications it may be important to estimate both the left and
  the right wave speeds $\lambda_1^-$ and $\lambda_3^+$. This is the
  case for instance in the so-called central-upwind schemes, see \eg
  \cite[\S3.1]{Kurganov_Noelle_Petrova_2001}, in the HLL, HLLC, HLLE
  approximate Riemann solvers, see \cite[p.~204,
  Rem~3.2]{Godlewski_Raviart_1996}, and in Arbitrary Lagrangian
  Eleurian hydrodynamics, see \eg
  \cite{Guermond_Popov_Saavedra_Yong_ALE_2016}.  It is possible to
  modify the Algorithm~\ref{update_algorithm} for this purpose as
  stated in the following corollary.
\begin{corollary}[Bounds on $\lambda_1^-$, $\lambda_3^+$]
  Consider Algorithm~\ref{update_algorithm} with line \ref{second_if}
  replaced by ``{\scriptsize $4{:}$} \textbf{\em if }
  $(\frac{v_{12}^k-v_{11}^k}{\lambda_{\min}^k}\le \epsilon$
  \textbf{\em and} $\frac{v_{32}^k-v_{31}^k}{\lambda_{\min}^k}\le
  \epsilon)$ \textbf{\em then}''
 and line \ref{output} changed to return $(v_{11}^k,v_{32}^k)$.  Then
 for every $\epsilon>0$, there exists an integer
 $k(u_L,a_L,u_R,a_R,\gamma,\epsilon)$ such that
 Algorithm~\ref{update_algorithm} terminates when
 $k=k(u_L,a_L,u_R,a_R,\gamma,\epsilon)$, and in this case $v_{11}^k$
 and $v_{32}^k$ satisfy the following guaranteed bounds:
\begin{equation}
0 \le \lambda_1^- - v_{11}^k \le \epsilon \lambda_{\max},\qquad
0 \le v_{32}^k - \lambda_3^+ \le \epsilon \lambda_{\max},
\end{equation}
\ie the error on $\lambda_1^-$ and $\lambda_3^+$ relative to $\lambda_{\max}$
is guaranteed to be bounded by $\epsilon$.
\end{corollary}

Note that in the above corollary $\lambda_1^-$ is approximated from
below and $\lambda_3^+$ is approximated from above. Hence, we always
guarantee that the exact Riemann fan is contained in the cone
$C^k:=\{(x,t) \st v_{11}^kt\le x\le v_{32}^k t, \ 0<t\}$ for any $k\ge
0$. Note also that $C^{k+1}\subset C^k$ for any $k\ge 0$.

\subsection{Gap condition}
The purpose of this section is to establish the following result,
which we call the gap condition.
\begin{lemma}[Gap condition] \label{Lem:gap} Given the left state
  $\bc_L:=(\rho_L,m_L,\mathcal{E}_L)$ and the right state
  $\bc_R:=(\rho_R,m_R,\mathcal{E}_R)$ of the one-dimensional
  Riemann problem \eqref{1D_Euler}, we have the following {\em gap}
  condition for the smallest and largest eigenvalues of the problem
\begin{equation}\label{Gap_ineq}
\lambda_3^{+}-\lambda_1^{-} \ge c(\gamma)(a_L+a_R)
\end{equation}
where  $a_L$, $a_R$ are the local sound speeds and $c(\gamma)$ is a constant defined by
\begin{equation}c(\gamma):=
\begin{cases}
\frac{2\sqrt{2(\gamma-1)}}{\gamma+1} &\text{if $\gamma\in(1,3]$,}\\
1 &\text{if $\gamma\in(3,+\infty)$.}
\end{cases}
\end{equation} 
\end{lemma}
\begin{proof} There are three possible cases for the solution of the
  Riemann problem.

  Case 1. The solution contains two rarefaction waves:
  $\phi(p_{\min})\ge 0$.  This implies that either there exists
  $p^*\ge 0$ such that $\phi(p^*)= 0$ or we have vacuum, \ie $0\le
  \phi(0)$.  If $\phi(p^*)= 0$ we derive
\[
u_R-u_L= \frac{2a_L(1-b\rho_L)}{\gamma-1}
\left(1-\left(\frac{p^*}{p_L}\right)^{\frac{\gamma-1}{2\gamma}
  }\right)+ \frac{2a_R(1-b\rho_R)}{\gamma-1}
\left(1-\left(\frac{p^*}{p_R}\right)^{\frac{\gamma-1}{2\gamma}
  }\right) \ge 0,
\]
and in the case of vacuum we get
\[
u_R-u_L \ge \frac{2}{\gamma-1} \left(a_L(1-b\rho_L) + a_R(1-b\rho_R)\right)>0.
\]
Using the fact that $p^*\le p_{\min}$, we derive from
\eqref{estimate_speed_1_Euler}--\eqref{estimate_speed_3_Euler} that
\[
\lambda_3^{+}-\lambda_1^{-} = u_R-u_L +a_L + a_R \ge a_L + a_R,
\]
which proves \eqref{Gap_ineq} with constant $c(\gamma)=1$ in this case.

Case 2. The solution contains one rarefaction and one shock wave:
$\phi(p_{\min}) < 0=\phi(p^*)\le\phi(p_{\max})$.  Then, we have that
$p_{\min}< p^*\le p_{\max}$ and
\[
0=\phi(p^*) = (p^*-p_{\min}) \sqrt{ \frac{A_{\min}}{p^*+B_{\min}} } +
\frac{2a_{\max}(1-b\rho_{\max})}{\gamma-1}
\left(\left(\frac{p^*}{p_{\max}}\right)^{\frac{\gamma-1}{2\gamma} }
  -1\right) + u_R-u_L
\]
where we recall that $A_{Z}:=\frac{2(1-b\rho_{Z})}{(\gamma+1)\rho_{Z}}$ and $B_{Z}:=\frac{\gamma-1}{\gamma+1}p_{Z}$.
Using the above we derive the following 
\begin{equation*}
\begin{aligned}
\lambda_3^{+}-\lambda_1^{-} &= 
a_{\max}\left( 
1 +  \frac{2(1-b\rho_{\max})}{\gamma-1} 
\big(1-\big(\frac{p^*}{p_{\max}}\big)^{\frac{\gamma-1}{2\gamma} }\big)\right)\\
&+
a_{\min}\left( 
\big(1 +  \frac{\gamma+1}{2\gamma} \frac{p^*-p_{\min}}{p_{\min}}\big)^{\frac12}
-\frac{(p^*-p_{\min})}{a_{\min}}\sqrt{ \frac{A_{\min}}{p^*+B_{\min}} }
\right)\\
&\ge a_{\max} + a_{\min}\Delta(p_{\min}).
\end{aligned}
\end{equation*}
where we define $\Delta(Z)$ by
\begin{equation}
\Delta(Z) = \big(1 +  \frac{\gamma+1}{2\gamma} \frac{p^*-p_{Z}}{p_{Z}}\big)^{\frac12}
-\frac{p^*-p_{Z}}{a_{Z}}\sqrt{ \frac{A_{Z}}{p^*+B_{Z}} }.
\end{equation}\label{Eq:DeltaZ}
We make a substitution $x=\frac{p^*}{p_{\min}}\ge 1$ and, with a an abuse of notation, 
we transform $\Delta(p_{\min})$ into 
\begin{equation}\label{Eq:delta_x_b}
\Delta(x)= \left(1 +  \frac{\gamma+1}{2\gamma} (x-1)\right)^{\frac12}
-\frac{(x-1)(1-b\rho_{\min})}{ (x+\frac{\gamma-1}{\gamma+1} )^{\frac12}}
 \sqrt{\frac{2}{\gamma(\gamma+1)}}.
\end{equation}
Using the property that $0\le 1-b\rho_{\min}\le 1$, we derive
\begin{equation}\label{Eq:delta_x}
\Delta(x)\ge \Delta_0(x) :=\left(1 +  \frac{\gamma+1}{2\gamma} (x-1)\right)^{\frac12}
-\frac{x-1}{ (x+\frac{\gamma-1}{\gamma+1} )^{\frac12}}
 \sqrt{\frac{2}{\gamma(\gamma+1)}}.
\end{equation}
We now want to find the minimum of the function $\Delta_0(x)$ in the interval $x\in [1,+\infty)$. We transform 
$\Delta_0(x)$ as follows
\[
\Delta_0(x)= \left( \frac{\gamma+1}{2\gamma} \right)^\frac12
\left(x+\frac{\gamma-1}{\gamma+1} \right)^{-\frac12}
\left(
x+\frac{\gamma-1}{\gamma+1} - (x-1)\frac{2}{\gamma+1}
\right)
\]
and after another substitution $y=x-1$, and another abuse of notation,
we have
\[
\Delta_0(y)= \left( \frac{\gamma+1}{2\gamma} \right)^\frac12
\left(y+\frac{2\gamma}{\gamma+1} \right)^{-\frac12}
\left(
\frac{\gamma-1}{\gamma+1} y + \frac{2\gamma}{\gamma+1} 
\right)= \frac{\gamma-1}{(2\gamma(\gamma+1))^\frac12}\psi(y)
\]
where
$\psi(y)=(y+\frac{2\gamma}{\gamma-1})\left(y+\frac{2\gamma}{\gamma+1}
\right)^{-\frac12}$.  The function $\psi(y)$ has a unique minimum on
the interval $[0,+\infty)$ at the point
$y_{\min}=\frac{2\gamma(3-\gamma)}{(\gamma-1)(\gamma+1)}$ provided
that $\gamma\le 3$.  The value of the minimum is
$\psi(y_{\min})=4\sqrt{\frac{\gamma}{(\gamma-1)(\gamma+1)}}$. If
$\gamma >3$ then the minimum is at $y=0$ and the value is
$\psi(0)=\frac{\sqrt{(\gamma+1)}}{\gamma-1}$.  Using the minimum value
of $\psi$ we get the following minimum of the function $\Delta_0(x)$
on the interval $[1,+\infty)$:
\[
\Delta_0(x)\ge \frac{2\sqrt{2(\gamma-1)}}{\gamma+1} \text{ if
  $\gamma\in(1,3]$,} \quad \text{and}\quad \Delta_0(x)\ge 1 \text{ if
  $\gamma\in(3,+\infty)$.}
\]
This finishes the proof in the second case because for the full range
of $\gamma$ we get $\Delta(x)\ge\Delta_0(x)\ge c(\gamma)$. This again
proves \eqref{Gap_ineq} since
$c(\gamma)\le 1$ for any $\gamma\in[1,\infty)$.

Case 3. The solution contains two shock waves: $\phi(p_{\min})
\le\phi(p_{\max}) < 0=\phi(p^*)$.  Then, we have that $p_{\min}\le
p_{\max}< p^*$ and
\begin{equation}
0=\phi(p^*) = 
(p^*-p_{\min}) \sqrt{ \frac{A_{\min}}{p^*+B_{\min}} } +(p^*-p_{\max}) \sqrt{ \frac{A_{\max}}{p^*+B_{\max}} } 
+ u_R-u_L.
\end{equation}
Similar to the previous case we derive 
\begin{equation*}
\begin{aligned}
\lambda_3^{+}-\lambda_1^{-} &= 
a_{\min}\left( 
\big(1 +  \frac{\gamma+1}{2\gamma} \frac{p^*-p_{\min}}{p_{\min}}\big)^{\frac12}
-\frac{(p^*-p_{\min})}{a_{\min}}\sqrt{ \frac{A_{\min}}{p^*+B_{\min}} }
\right)\\
&+
a_{\max}\left( 
\big(1 +  \frac{\gamma+1}{2\gamma} \frac{p^*-p_{\max}}{p_{\max}}\big)^{\frac12}
-\frac{(p^*-p_{\max})}{a_{\max}}\sqrt{ \frac{A_{\max}}{p^*+B_{\max}} }
\right)\\
&\ge a_{\max} \Delta(p_{\min})+ a_{\min}\Delta(p_{\min})
\end{aligned}
\end{equation*}
where $\Delta(Z)$ is the same as before, see \eqref{Eq:DeltaZ}. 
We now use the fact that $\Delta(Z)\ge c(\gamma)$ 
when $p_Z\le p^*$ to finish the proof in this case.
\end{proof}

\section{Numerical illustrations} \label{Sec:numerical} We illustrate
the performances of Algorithm~\ref{update_algorithm} in this
section. We only consider test cases where there is at least one
shock, since the cases with two expansion waves are trivial. The set
of test problems we use is based on the performance tests given in
\cite[Section 4.3.3]{Toro_2009}.  The code that we used is included in
Appendix~\ref{Sec:source_code}.

\subsection{Fast expansion and slow shock}
Algorithm~\ref{update_algorithm} may terminate and give an estimate on
$\lambda_{\max}$ with the required accuracy before $p^*$ is estimated
correctly. This situation may happen when one of the two extreme waves
is a fast expansion (rarefaction) and the other wave is a slow shock. To illustrate
this effect, let us assume for instance that the left wave is a fast
expansion and the right wave is the slow shock, say $p_L>p^*>p_R$ and
$(\lambda_1^-)_{-} >( \lambda_3^-)_{+}=(\lambda_3^+)_+$. Note that in
this case we always have $p_R\le p_1^k\le p_2^k\le p_L$ for any
$k\ge 0$; hence, $v_{11}^k = v_{12}^k= u_L -a_L = \lambda_1^-$ for any
$k\ge 0$.  At some point in the algorithm there will be an iteration
level $k$ such that both $p_1^k$ and $p_2^k$ are close enough to $p^*$
so that
$0 \le v_{32}^k-\lambda_3^+ \le (\lambda_1^-)_- -(\lambda_3^+)_+$ and
$0 \le v_{31}^k-\lambda_3^+ \le (\lambda_1^-)_- -(\lambda_3^+)_+$.
Hence, using that $x\ge y$ implies that $x-y + y_+ -x_+\ge 0$, we have
\begin{align*}
(v_{11}^k)_- -  (v_{32}^k)_+ &= (\lambda_1^-)_- - v_{32}^k + v_{32}^k-(v_{32}^k)_+\\
& \ge (\lambda_3^+)_+ - \lambda_3^++ v_{32}^k -(v_{32}^k)_+\ge 0.
\end{align*}
This means that  $\lambda_{\max}^k :=
\max((v_{32}^k)_+,(v_{11}^k)_-)=(v_{11}^k)_-=(\lambda_1^-)_-$. Likewise
\begin{align*}
(v_{12}^k)_- -  (v_{31}^k)_+ &= (\lambda_1^-)_- - v_{31}^k + v_{31}^k-(v_{31}^k)_+\\
& \ge (\lambda_3^+)_+ - \lambda_3^++ v_{31}^k -(v_{31}^k)_+\ge 0,
\end{align*}
\ie 
$\lambda_{\min}^k := \max((v_{31}^k)_+,(v_{12}^k)_-)_{+}
=(v_{12}^k)_-= (\lambda_1^-)_-$.
In conclusion at iteration $k$, we have
$\frac{\lambda_{\max}^k}{\lambda_{\min}^k} -1 =0$; in other word the
algorithm stops irrespective of the tolerance, and $\lambda_{\max}^k =
\lambda_1^-$ but $p_1^k$ and $p_2^k$ may still be far from $p^*$.
To illustrate this phenomenon we consider the following test cases:
\begin{table}[H]
\centering\begin{tabular}{||c|c|c|c|c|c|c||}\hline
case &$\rho_L$ & $\rho_R$ & $u_L$ & $u_R$ & $p_L$ & $p_R$ \\
\hline
1 & 1.0 & 1.0 & 0.0 & 0.0 & 0.01 & 100.0 \\ \hline
2 & 1.0 & 1.0 & -1.0 & -1.0 & 0.01 & 100.0 \\ \hline
3 & 1.0 & 1.0 & -2.18 & -2.18 & 0.01 & 100.0 \\\hline
\end{tabular}\vspace{-0.5\baselineskip}
\end{table}
We run the algorithm with $\epsilon=10^{-15}$. The results are
\begin{table}[H]
\centering\begin{tabular}{||c|c|c|c||}\hline
case & $k$ & $\lambda_{\max}^k$ & $\lambda_{\max}$ \\ \hline
1      &  0    &  11.83215956619923 & 11.83215956619923  \\ \hline
2      &  1    &  10.83215956619923 & 10.83215956619923   \\ \hline
3      &  2    &    \ 9.65215956619923 & \ 9.65215956619923  \\ \hline
\end{tabular}
\centering\begin{tabular}{||c|c|c|c|c||}\hline
case & $k$ &  $p_1^k$ & $p_2^k$& $p^*$ \\ \hline
1      &  0    &  37.70559999364363&82.98306927558072 & 46.09504424886797\\ \hline
2      &  1    &  45.87266091833658&46.70007404915459 & 46.09504424886797\\ \hline
3      &  2    &  46.09504109404150&46.09505272562230 & 46.09504424886797 \\ \hline
\end{tabular}\vspace{-0.5\baselineskip}
\end{table}
In the first case, The algorithm terminates just after the
initialization, \ie at $k=0$, and gives the exact value of
$\lambda_{\max}$ up to a rounding error, but it gives
$p_1^0< p^*< p_2^0$.  In the second case the algorithm terminates at
$k=1$ with the exact value of $\lambda_{\max}$ up to a rounding
error, but it gives $p_1^1< p^*< p_2^1$.  By biasing the problem a
little bit more to the right, \ie by taking $u_L=u_R=2.18$ , the
algorithm terminates at $k=2$ and gives again the exact value of
$\lambda_{\max}$ up to a rounding error, but $p_1^2< p^*< p_2^2$. We
have verified that for $u_L=u_R\ge 2.2$ the right-moving shock wave is
the fastest and the algorithm always terminates at $k=3$ and gives
$p_1^k=p^*=p_2^k$ up to round-off errors.


\subsection{Fast shock} \label{Sec:fast_shock} The most demanding
situation happens when the fastest wave is a shock, since in this case
the algorithm must find $p^*$ up to the assigned tolerance to
terminate. We consider the following two cases introduced in
\cite[Section 4.3.3]{Toro_2009}:
\begin{table}[H]
\centering\begin{tabular}{||c|c|c|c|c|c|c||}\hline
case &$\rho_L$ & $\rho_R$ & $u_L$ & $u_R$ & $p_L$ & $p_R$ \\ \hline
1 & 1.0 & 1.0 & 10.0 & 10.0 & 1000.0 & 0.01 \\ \hline
2 &5.99924 &5.99242&19.5975&-6.19633&460.894&46.0950\\\hline
\end{tabular}\vspace{-0.5\baselineskip}
\end{table}
\noindent 
In case 1, the left wave is a rarefaction and the right wave
is a shock. In case 2, both waves are shocks. We run the algorithm
with various values of $\epsilon$; the results are
\begin{table}[H]
\centering\begin{tabular}{||c|c|c|c|c|c||}\hline
 &$\epsilon$   & $k$  & $\lambda_{\max}^k$      & $p_1^k$                         & $p_2^k$ \\ \hline
 1&$10^{-1}$     &  1    &  33.81930602421521  & 455.2466713625296   &  472.7977828960125 \\ \hline
 1&$10^{-2}$     &  2    &  33.51755796979217  & 460.8933865271423   &  460.8946107187795 \\ \hline 
 1&$10^{-15}$   &  3    &  33.51753696690324  & 460.8937874913834   &  460.8937874913835 \\ \hline \hline
 2&$10^{-1}$     & 1     &  12.25636731290528 &  1691.520678281327   &  1692.676852734373 \\ \hline
 2&$10^{-4}$     & 2     &  12.25077812313116 &  1691.646955398068   &  1691.646955407751 \\ \hline
 2&$10^{-15}$   & 3      & 12.25077812308434 &  1691.646955399126    &  1691.646955399126 \\ \hline
\end{tabular}\vspace{-0.5\baselineskip}
\end{table}
We observe that the algorithm converges very fast and it takes three
steps to reach $10^{-15}$ accuracy on $\lambda_{\max}$. These two
examples are representative of all the tests we have done in that most
of the times the tolerance $10^{-15}$ is achieved in at most three
steps.

When running the code $1,000,000$ times on case 2 with $10^{-15}$
tolerance, which amounts to three iterations per case, the total CPU
time was $0.736$ seconds on a machine with the following
characteristics: Intel(R) Xeon(R) CPU E3-1220 v3 \@ 3.10GHz.

\subsection{Overhead estimation} \label{Sec:overhead}
In order to give an idea of the overhead cost of estimating the
maximum wave speed with the above algorithm, we test the code on the
so-called Sod shocktube and Leblanc shocktube.  The data for these two
Riemann problems are recalled in the following table:
\begin{table}[H]
\centering\begin{tabular}{||c|c|c|c|c|c|c|c||}\hline
case   &$\gamma$     &$\rho_L$ & $\rho_R$ & $u_L$ & $u_R$ & $p_L$ & $p_R$ \\ \hline
Sod    &$7/5$        & 1.0     & 0.125    & 0.0  & 0.0   & 1.0   & 0.1 \\ \hline
Leblanc&$5/3$        & 1.0     &0.001     & 0.0  & 0.0   & 0.1   & $10^{-10}$\\\hline
\end{tabular}\vspace{-0.5\baselineskip}
\end{table}
The problems are set in two space dimensions in the domain
$(0,0.1)\CROSS (0,1)$ and are solved by using a finite element
technique described in \cite{Guermond_Popov_Hyp_2015}. It is an
explicit method using continuous finite elements on unstructured grids
that works in any space dimension. The time stepping is done by using
the SPP RK(3,3) algorithm.  The two-dimensional triangular mesh that
is used for this test is shown on the top panel of
Figure~\ref{Fig:Sod_Leblanc_mesh}. The mesh is composed of $2400$
triangles and $1311$ vertices. The stopping time and number of time
sub-steps performed are $t=0.4$ and $1170$ for the Leblanc test case,
respectively, and $t=0.2$ and $1101$ for the Sod test case,
respectively.  For each time sub-step we use the code whose source is
given in \S\ref{Sec:source_code} to estimate the maximum wave speed.
This code is an exact transcription in Fortran~95 of
Algorithms~\ref{initialization_algorithm} and \ref{update_algorithm}.
The code is called two times for every edge in the mesh at every Runge
Kutta substep. For each sub-step we record the number of times the
maximum wave speed code is called, say $N_{\text{tot}}$, and the
number of times the algorithm enters the iterative loop, say
$N_{\text{iter}}$, \ie the number of times the algorithm goes beyond
the exit test in line~\ref{Exit_alg1} in
Algorithm~\ref{initialization_algorithm} times the number of
iterations in the iterative loop.  We call overhead the ratio
$N_{\text{iter}}/N_{\text{tot}}$; note that this number is greater
than or equal to $1$ if one never exits at line~\ref{Exit_alg1} in
Algorithm~\ref{initialization_algorithm}.  We show in the left and
right panels of Figure~\ref{Fig:Sod_Leblanc_mesh} the overhead as a
function of the time index for various values of the tolerance
$\epsilon$ (denoted ``tol'' in the figure keys). We have observed that
the overhead is $0$ for the Sod test for $\epsilon\ge 6\CROSS 10^{-3}$
(data not shown here). In both the Sod and Leblanc cases we observe
that the overhead decreases rapidly after a few time sub-steps.  The
overhead is at most $1\%$ for the Sod test at the beginning and is
less that $0.1\%$ after 100 time sub-steps; the overhead does not to
depend a lot on $\epsilon$ when $\epsilon\le 5\CROSS 10^{-4}$. For the
Leblanc test, the overhead is at most $10\%$ at the beginning and is
less than $1\%$ after 10 substeps. There seems to be a weak dependency
on $\epsilon$ in this test, but the overhead can be kept well below
$1\%$ for $\epsilon \le 10^{-4}$.  

Since for any practical purpose it is enough to take $\epsilon \in
[10^{-3},5\CROSS 10^{-2}]$ to obtain an accurate estimate of the
maximum wave speed in the algorithms described in
\citep{Guermond_Popov_Hyp_2015,Guermond_Popov_Saavedra_Yong_ALE_2016},
the above tests show that the overhead of the proposed method is
virtually less than $1\%$.

\begin{figure}\centering
\includegraphics[width=0.8\textwidth,viewport=83 345 502 416,clip=true]{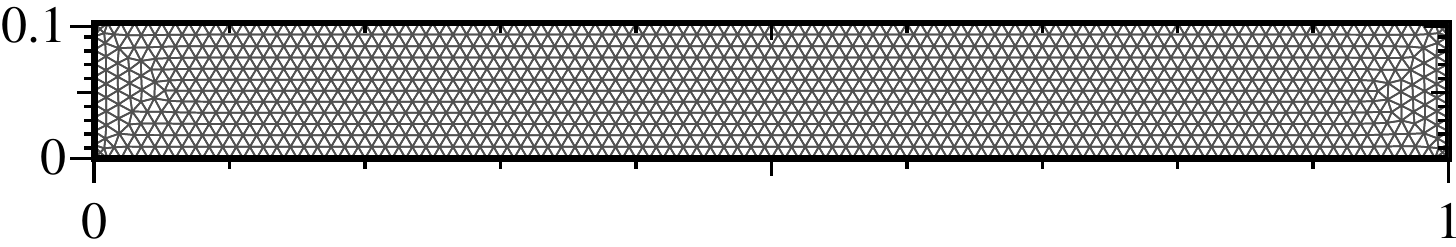}\\
\includegraphics[width=0.4\textwidth,viewport=20 10 352 245,clip=true]{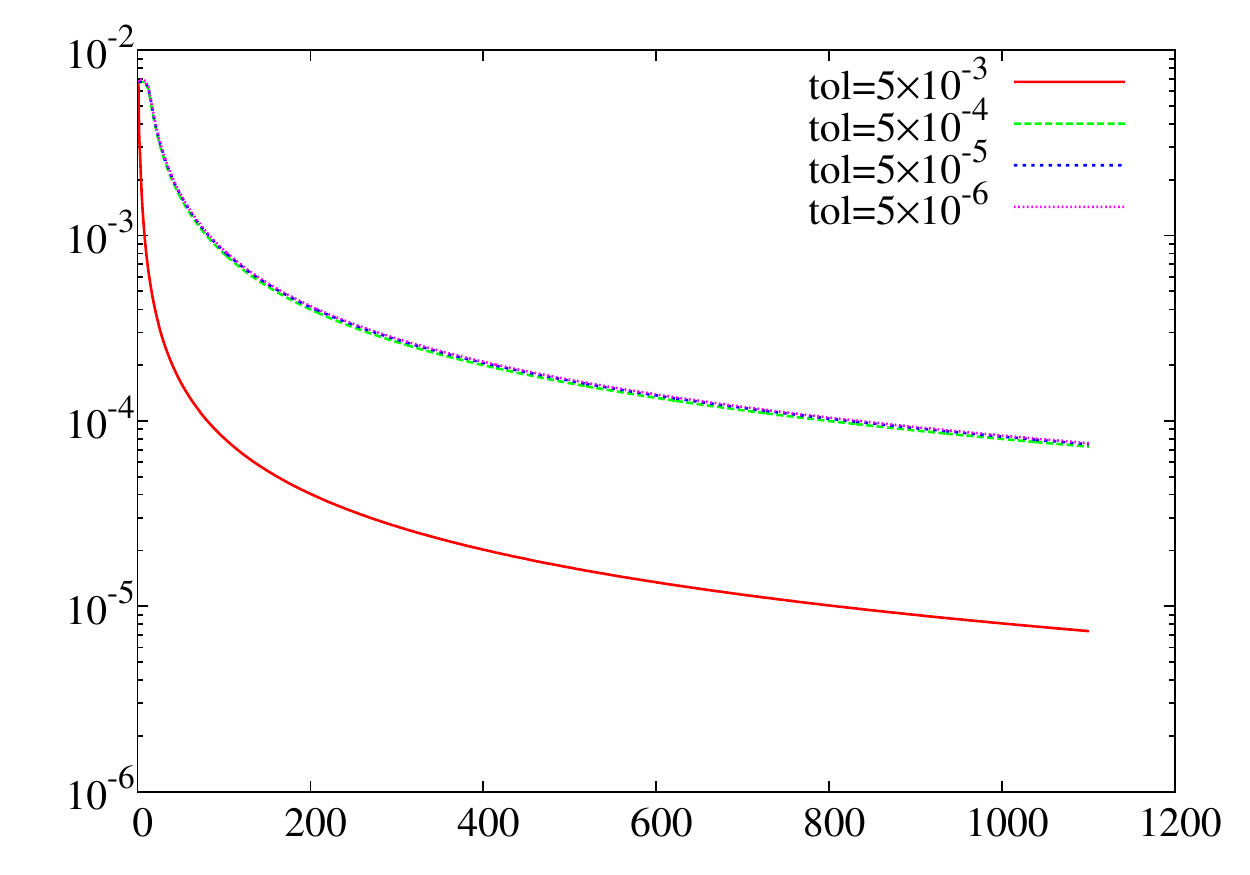}%
\includegraphics[width=0.4\textwidth,viewport=20 10 352 245,clip=true]{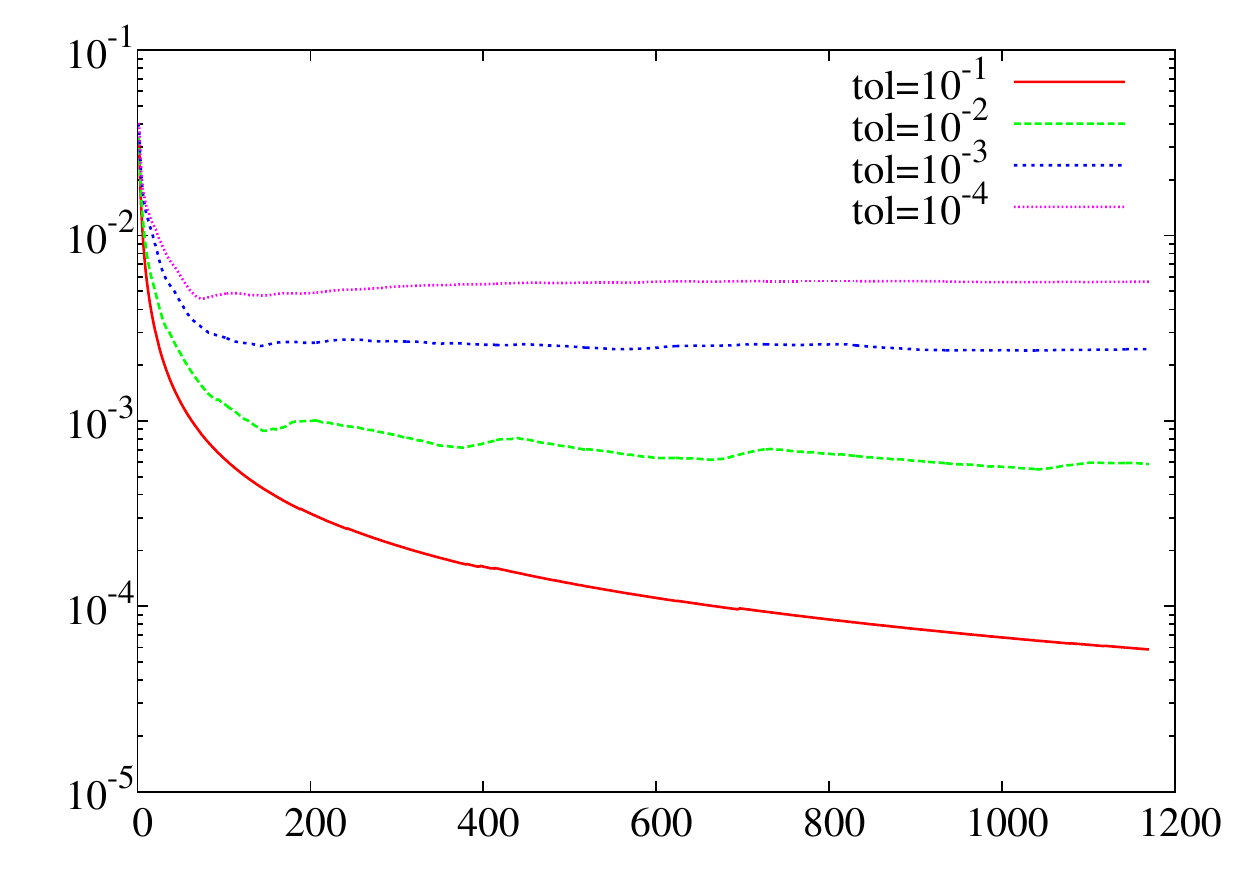}%
\caption{Overhead $N_{\text{\em iter}}/N_{\text{\em tot}}$ vs. time iterations for 
the Sod shocktube (left panel) and the Leblanc shocktube
  (right panel) on the Delaunay mesh shown on the top panel.}
\label{Fig:Sod_Leblanc_mesh}
\end{figure}

\begin{appendix}
\section{Co-volume equation of state} \label{Sec:Appendix}
We verify in this appendix the following statement. 
\begin{proposition} \label{Prop:co_volume_admissible}
  If the left and right states in the Riemann problem \eqref{1D_Euler}
  are such that $0<1-b\rho_L,1-b\rho_R < 1$, then the exact solution
  of the Riemann problem satisfies $0< 1-b\rho \le 1$.
\end{proposition} 
\begin{proof} 
  We split the proof into five parts where we analyze
  the solution across each wave.

 {(1) Left rarefaction wave.} Assume
  that the left wave is a rarefaction, then the wave speed, $S:=u - a
  = u - \sqrt{\frac{\gamma p}{\rho(1-b\rho)} }$, should increase along
  the wave from left to right. Using that both the specific entropy
  $s=\log(e^{\frac{1}{\gamma-1}} (\frac1\rho -b))$ and the generalized
  left Riemann invariant $u+ \frac{2a}{\gamma-1}(1-b\rho)$ are
  constant across the left rarefaction, we obtain (see Eq. (4.93) and
  Eq. (4.94) in \cite[Chapter 4.7]{Toro_2009}):
\[
S(p)=u_L+\frac{2(1-b\rho_L)}{\gamma-1}a_L-\frac{2(1-b\rho+\frac{\gamma-1}{2})}{(\gamma-1)(1-b\rho)}
a_L(1-b\rho_L)\left(\frac{p}{p_L}\right)^{\frac{\gamma-1}{2\gamma}}
\]
where
$\frac1{\rho}-b=(\frac1{\rho_L}-b)\left(\frac{p_L}{p}\right)^{\frac{1}{\gamma}}$.
Owing to the assumption $\frac1{\rho_L}-b>0$, we conclude that
$\rho=\rho(p)$ is an increasing function of $p$. After some
computation, we also prove that the speed $S(p)$ is a decreasing
functions of $p$.  Therefore the rarefaction wave can be parametrized
by $p$ as a decreasing parameter from $p_L$ to $p^*$, \ie the left
wave is well defined. Hence we have $S(p_L)\le S(p)$ for $p^*< p \le
p_L$, which confirms that we have a rarefaction wave.  Finally, using
that $\rho(p)$ is a decreasing function of $p$, we conclude that
$0<1-b\rho_L<1-b\rho$. Observing also that
$\rho=\rho_L\frac{1-b\rho_L}{1-b\rho}\left(\frac{p}{p_L}\right)^{\frac{1}{\gamma}}$
we conclude that $\rho\in [0,\rho_L]$ since $p^*\ge 0$ and $p\in
[p^*,p_L]$.

{(2) Left shock wave.} Assume that the
left wave is a shock from the left state $(\rho_L,u_L,p_L)$ to the
state $(\rho,u,p)$. The Rankine--Hugoniot condition combined with the equation of state implies that 
(see \eg \cite[Chapter 4.7]{Toro_2009}) 
\[
\rho=
\frac{\rho_L(\frac{p}{p_L}+\frac{\gamma-1}{\gamma+1})}
{\frac{\gamma-1+2b\rho_L}{\gamma+1}\frac{p}{p_L}+\frac{\gamma+1-2b\rho_L}{\gamma+1}}.
\]
For details, we refer to \cite[(4.89) on p.145]{Toro_2009}. We
introduce $y:=b\rho$, $y_L:=b\rho_L$ and $\beta:=\frac{p}{p_L}$ and
the above equality can be re-written as follows:
\[
y=y_L\,\frac{\beta+\frac{\gamma-1}{\gamma+1}}{\frac{\gamma-1}{\gamma+1}\beta+1+\frac{2y_L}{\gamma+1}(\beta-1)},
\]
which is equivalent to 
\[
1- y = \frac{1}{1+\frac{2y_L(\beta-1)}{(\gamma-1)\beta+\gamma+1}} (1-y_L).
\]
Hence, we conclude that provided that $0\le y_L<1$ and $\beta\ge 1$,
then $y_L\le y < 1$. This proves the result in the second case.

{(3) Right rarefaction wave.} The proof
is analogous to the case of the left rarefaction wave.

{(4) Right shock wave.} The proof is
analogous to the case of the left shock wave.

{(5) Contact wave.}  The state on the left of the contact wave is
the right state from the left wave which has already been proved to be
admissible in 1--4. Similarly, the state on the right of the contact
wave is the left state from the right wave which has also been proved
to be admissible in 1--4. This completes the proof.
\end{proof}

\section{A counter-example} \label{Sec:counter_example}
We show in this section that taking $\max(|u_L|+a_L,|u_R|+a_R)$ as an
estimate of the maximum wave speed in the Riemann problem, as it is
frequently done in the literature, can actually underestimate the
actual maximum wave speed.

For instance take $u_L=u_R=0$ and select $p_L$ and $p_R$ so that
$p_L/p_R$ is a number less than $1$. Then
$a_R = a_L \sqrt{\frac{p_R}{p_L}} \sqrt{\frac{\rho_L}{\rho_R}}$. Now
we choose $\frac{\rho_L}{\rho_R}$ so that
$\sqrt{\frac{p_R}{p_L}} \sqrt{\frac{\rho_L}{\rho_R}}<1$; note that we
can make this number as small as we want. Then
$\max(|u_L|+a_L,|u_R|+a_R) = a_L$.  But $u_L=u_R=0$ and $p_L< p_R$
implies that $\phi(p_L)<0$ and $\phi(p_R)>0$; whence $p_L< p^*<p_R$.
Therefore the maximum wave speed is the absolute value of the left
speed given in \eqref{estimate_speed_1_Euler},
$a_L(1+\frac{\gamma+1}{2\gamma}(\frac{p^*-p_L}{p_L}))^{\frac12}$,
which is strictly larger than $a_L$, whence the conclusion.

To illustrate the above argument we now give two examples.  First we
consider case 2 from \S\ref{Sec:fast_shock}. The 1-wave and the 3-wave
are both shocks. The correct wave speed is
$\lambda_{\max}\approx 12.25$ but the traditional estimate
gives $\max(|u_L|+a_L,|u_R|+a_R) \approx 29.97$, which is
clearly an overestimate of $\lambda_{\max}$; the ratio is
approximately $0.41$. Second we consider the following two states
\[
\rho_L=0.01,\ \rho_R= 1000, \quad u_L=0, \ u_R= 0, \quad p_L= 0.01, \ p_R = 1000.
\]
We obtain $\lambda_{\max}\approx 5.227$ and $\max(|u_L|+a_L,|u_R|+a_R)
\approx 1.183$. It is clear that the heuristic estimate is far from
the real value; the ratio is approximately $4.4$.  In conclusion the
estimate $\max(|u_L|+a_L,|u_R|+a_R)$ is grossly unreliable.

\section{Source code} \label{Sec:source_code}
{\tiny
\begin{lstlisting}
! Authors: Jean-Luc Guermond and Bojan Popov, Texas A&M, April 5, 2016
MODULE lambda_module
  PUBLIC                           :: lambda
  INTEGER,      PRIVATE, PARAMETER :: Mgas=5
  REAL(KIND=8), PRIVATE, PARAMETER :: gamma=(Mgas+2.d0)/Mgas, exp=1.d0/(Mgas+2), b=0.0d0
  REAL(KIND=8), PRIVATE            :: al, capAl, capBl, covl, ar, capAr, capBr, covr
CONTAINS

  SUBROUTINE init(rhol,pl,rhor,pr)
    IMPLICIT NONE
    REAL(KIND=8), INTENT(IN) :: rhol, pl, rhor, pr
    al = SQRT(gamma*pl/rhol)
    capAl = 2/((gamma+1)*rhol)
    capBl = pl*(gamma-1)/(gamma+1)
    covl = SQRT(1-b*rhol)
    ar = SQRT(gamma*pr/rhor)
    capAr = 2/((gamma+1)*rhor)
    capBr = pr*(gamma-1)/(gamma+1)
    covr = SQRT(1-b*rhor)
  END SUBROUTINE init

  SUBROUTINE lambda(tol,rhol,ul,pl,rhor,ur,pr,lambda_max,pstar,k)
    IMPLICIT NONE
    REAL(KIND=8), INTENT(IN) :: tol, rhol, ul, pl, rhor, ur, pr
    REAL(KIND=8), INTENT(OUT):: lambda_max, pstar
    INTEGER,      INTENT(OUT):: k
    REAL(KIND=8)             :: capAmin, capBmin, acovmin, acovmax, ratio
    REAL(KIND=8)             :: lambda_min, phimax, phimin, ptilde
    REAL(KIND=8)             :: phi1, phi11, phi12, phi112, phi2, phi22, phi221
    REAL(KIND=8)             :: p1, p2 , pmin, pmax, rhomin, rhomax, v11, v12, v31, v32
    !===Initialization
    CALL init(rhol,pl,rhor,pr)
    k = 0
    IF (pl.LE.pr) THEN
       pmin   = pl
       rhomin = rhol
       pmax   = pr
       rhomax = rhor
    ELSE
       pmin   = pr
       rhomin = rhor
       pmax   = pl
       rhomax = rhol
    END IF
    capAmin = 2/((gamma+1)*rhomin)
    capBmin = pmin*(gamma-1)/(gamma+1)
    acovmin = SQRT(gamma*pmin*(1-b*rhomin)/rhomin)
    acovmax = SQRT(gamma*pmax*(1-b*rhomax)/rhomax)
    ratio = (pmin/pmax)**exp
    phimin = (2/(gamma-1.d0))*acovmax*(ratio-1.d0) + ur-ul
    IF (phimin.GE.0) THEN
       pstar = 0.d0
       lambda_max = MAX(MAX(-(ul-al),0.d0), MAX(ur+ar,0.d0))
       RETURN
    END IF
    phimax = (pmax-pmin)*SQRT(capAmin/(pmax+capBmin)) + ur-ul
    ptilde = pmin*((acovmin+acovmax - (ur-ul)*(gamma-1)/2)&
         /(acovmin + acovmax*ratio))**(Mgas+2)
    IF (phimax < 0.d0) THEN
       p1 = pmax
       p2 = ptilde
    ELSE
       p1=pmin
       p2 = MIN(pmax,ptilde)
    END IF
    !===Check for accuracy after initialization
    v11 = lambdaz(ul,pl,al/covl,p2,-1)
    v12 = lambdaz(ul,pl,al/covl,p1,-1)
    v31 = lambdaz(ur,pr,ar/covr,p1,1)
    v32 = lambdaz(ur,pr,ar/covr,p2,1)
    lambda_max = MAX(MAX(v32,0.d0),MAX(-v11,0.d0))
    lambda_min = MAX(MAX(MAX(v31,0.d0),MAX(-v12,0.d0)),0.d0)
    IF (lambda_min>0.d0) THEN
       IF (lambda_max/lambda_min -1.d0 .LE. tol) THEN
          pstar = p2
          RETURN
       END IF
    END IF
    p1 = MAX(p1,p2-phi(p2,ul,pl,ur,pr)/phi_prime(p2,pl,pr))
    !===Iterations
    DO WHILE(.TRUE.) 
       v11 = lambdaz(ul,pl,al/covl,p2,-1)
       v12 = lambdaz(ul,pl,al/covl,p1,-1)
       v31 = lambdaz(ur,pr,ar/covr,p1,1)
       v32 = lambdaz(ur,pr,ar/covr,p2,1)
       lambda_max = MAX(MAX(v32,0.d0),MAX(-v11,0.d0))
       lambda_min = MAX(MAX(MAX(v31,0.d0),MAX(-v12,0.d0)),0.d0)
       IF (lambda_min>0.d0) THEN
          IF (lambda_max/lambda_min -1.d0 .LE. tol) THEN
             pstar = p2
             RETURN
          END IF
       END IF
       phi1 =  phi(p1,ul,pl,ur,pr)
       phi11 = phi_prime(p1,pl,pr)
       phi2 =  phi(p2,ul,pl,ur,pr)
       phi22 = phi_prime(p2,pl,pr)
       IF (phi1>0.d0) THEN
          lambda_max = lambda_min
          RETURN
       END IF
       IF (phi2<0.d0) RETURN
       phi12 = (phi2-phi1)/(p2-p1) 
       phi112 = (phi12-phi11)/(p2-p1)
       phi221 = (phi22-phi12)/(p2-p1)
       p1 = p1 - 2*phi1/(phi11 + SQRT(phi11**2 - 4*phi1*phi112))
       p2 = p2 - 2*phi2/(phi22 + SQRT(phi22**2 - 4*phi2*phi221))
       k = k+1
    END DO
  END SUBROUTINE lambda

  FUNCTION lambdaz(uz,pz,az,pstar,z) RESULT(vv)
    IMPLICIT NONE
    REAL(KIND=8), INTENT(IN) :: uz,pz,az,pstar
    INTEGER,      INTENT(IN) :: z
    REAL(KIND=8)             :: vv
    vv = uz + z*az*SQRT(1+MAX((pstar-pz)/pz,0.d0)*(gamma+1)/(2*gamma))
  END FUNCTION lambdaz

  FUNCTION phi(p,ul,pl,ur,pr) RESULT(vv)
    IMPLICIT NONE
    REAL(KIND=8), INTENT(IN) :: p, ul, pl, ur, pr
    REAL(KIND=8)             :: vv, fl, fr
    IF (p>pl) THEN
       fl = (p-pl)*SQRT(capAl/(p+capBl))
    ELSE
       fl = (2*al/(gamma-1))*((p/pl)**exp-1)
    END IF
    IF (p>pr) THEN
       fr = (p-pr)*SQRT(capAr/(p+capBr))
    ELSE
       fr = (2*ar/(gamma-1))*((p/pr)**exp-1)
    END IF
    vv = fl*covl + fr*covr + ur - ul
  END FUNCTION phi

  FUNCTION phi_prime(p,pl,pr) RESULT(vv)
    IMPLICIT NONE
    REAL(KIND=8), INTENT(IN) :: p, pl, pr
    REAL(KIND=8)             :: vv, fl, fr
    IF (p>pl) THEN
       fl = SQRT(capAl/(p+capBl))*(1-(p-pl)/(2*(capBl+p)))
    ELSE
       fl = (al/(gamma*pl))*(p/pl)**(-(gamma+1)/(2*gamma))
    END IF
    IF (p>pr) THEN
       fr = SQRT(capAr/(p+capBr))*(1-(p-pr)/(2*(capBr+p)))
    ELSE
       fr = (ar/(gamma*pr))*(p/pr)**(-(gamma+1)/(2*gamma))
    END IF
    vv = fl*covl + fr*covr
  END FUNCTION phi_prime
END MODULE lambda_module

PROGRAM riemann
  USE lambda_module
  IMPLICIT NONE
  REAL(KIND=8) :: rhol, ul, pl, rhor, ur, pr, lambda_max, pstar, tol, t1, t2
  INTEGER      :: k, n
  OPEN(UNIT = 21, FILE = 'data', FORM = 'formatted', STATUS = 'unknown')
  READ(21,*) rhol, rhor, ul, ur, pl, pr
  READ(21,*) tol
  CLOSE(21)
  CALL CPU_TIME(t1)
  DO n = 1, 1000000
     CALL lambda(tol, rhol, ul, pl, rhor, ur, pr, lambda_max, pstar, k)
  END DO
  CALL CPU_TIME(t2)
  WRITE(*,*) ' CPU ', t2-t1
  WRITE(*,'(2(A,e23.17,x),A,I1)') 'lambda_max=', lambda_max, 'pstar=', pstar, 'k=', k 
END PROGRAM riemann
\end{lstlisting}
}
\end{appendix}

\bibliographystyle{abbrvnat} \bibliography{ref}

\end{document}